\documentclass[twoside,12pt,intlimits]{amsart}

\newcommand{\new}[1]{#1}
\newcommand{\newer}[1]{#1}

\usepackage{amsmath,amssymb,amsthm}
\usepackage{amsfonts}
\usepackage{amsxtra}
\usepackage{bbm}
\usepackage[usenames,dvipsnames]{xcolor}

\usepackage{xspace}
\usepackage{graphicx}

\usepackage{enumitem}
\setenumerate{label={\rm (\alph{*})}}

\usepackage[frak,mathbb,operator,jump]{paper_diening}
\usepackage{ifthen}
\usepackage{graphicx}

\newtheorem{theorem}{Theorem}[section]
\newtheorem{algorithm}{Algorithm}[section]
\newtheorem{lemma}[theorem]{Lemma}
\newtheorem{proposition}[theorem]{Proposition}
\newtheorem{corollary}[theorem]{Corollary} 

\newtheorem{definition}[theorem]{Definition}

\newtheorem{remark}[theorem]{Remark}

\definecolor{lightgrey}{rgb}{.7,.7,.7}

\definecolor{darkgreen}{rgb}{0,.5,0}

\numberwithin{equation}{section}

\newcommand{\meantmp}[2]{#1\langle{#2}#1\rangle}

\newcommand{\osc}{{\ensuremath{\mathrm{osc}}}}

\newcommand{\refi}[2]{\ensuremath{{\textsf{\rm Ref}(#1; #2)}}}
\newcommand{\refine}[2]{\ensuremath{{\textsf{\rm Refine}(#1; #2)}}}
\newcommand{\refd}[2]{\ensuremath{{\textsf{\rm ref'd}(#1; #2)}}}
\newcommand{\midpoint}[1]{\ensuremath{{\textsf{\rm midpt}(#1)}}}
\newcommand{\Popt}[1][]{\Pop^{\rm opt}_{#1}}
\newcommand{\Gopt}[1][]{\mathcal{G}^{\rm opt}_{#1}}

\newcommand{\ancestor}{{\tt anc}}

\newcommand{\descendant}{{\tt desc}}

\newcommand{\children}{{\tt child}}
\newcommand{\parents}{{\tt parent}} 
\newcommand{\generation}{{\tt gen}}

\newcommand{\childof}{\rhd} 
\newcommand{\parentof}{\lhd} 
\newcommand{\ancestorof}{{<}\mspace{-8mu}{\lhd}}
\newcommand{\descendantof}{{\rhd}\mspace{-8mu}{>}}{

\newcommand{\tildeMk}[1][k]{\widetilde{\mathcal{M}}_{#1}}
\newcommand{\Mk}[1][k]{\mathcal{M}_{#1}}

\newcommand{\partnerof}{{\circ{}}\hspace{-0.20em}{\circ{}}}
\newcommand{\free}{{\tt free}}
\newcommand{\cGD}{\ensuremath{c_{\rm GD}}}

\newcommand{\bbThat}{\widehat{\bbT}}
\newcommand{\bbPhat}{\widehat{\bbP}}
\newcommand{\dx}{\ensuremath{\mathrm{d}x}}
\newcommand{\ds}{\mathrm{d}s\xspace}

\newcommand{\VT}[1][\mathcal{T}]{\ensuremath{\mathbb{V}(#1)}}
\newcommand{\VoT}[1][\mathcal{T}]{\ensuremath{\mathbb{V}_0(#1)}}

\newcommand{\tria}{\ensuremath{\mathcal{T}}}
\newcommand{\plus}{^{++}}
\newcommand{\vT}[1][\tria]{\ensuremath{v_{#1}}}
\newcommand{\uT}[1][\tria]{\ensuremath{u_{#1}}}
\newcommand{\sides}{\ensuremath{\mathcal{S}}}
\newcommand{\nodes}{\ensuremath{\mathcal{N}}}
\newcommand{\estP}[1][\Pop]{\ensuremath{\mathcal{E}_{#1}}}
\newcommand{\estT}[1][\tria]{\ensuremath{\mathcal{E}_{#1}}}
\newcommand{\estPbar}[1][\Pop]{\ensuremath{\bar{\mathcal{E}}_{#1}}}
\newcommand{\estTbar}[1][\tria]{\ensuremath{\bar{\mathcal{E}}_{#1}}}

\newcommand{\EstPbar}[1][\Pop]{\ensuremath{\bar{{E}}_{#1}}}

\newcommand{\Cleq}{\ensuremath{\lesssim}}
\newcommand{\Cgeq}{\ensuremath{\gtrsim}}

\newcommand{\supp}{\operatorname{supp}}

\newcommand{\Pop}{\mathcal{P}}

{~\ttfamily\begin{samepage}%
\begin{tabbing}%
 \hspace*{5mm}\=\hspace{3ex}\=\hspace{3ex}\=\hspace{3ex}\=\hspace{3ex}%
\=\hspace{3ex}\=\hspace{3ex}\=\hspace{3ex}\=\hspace{3ex}\kill}%
{\end{tabbing}\end{samepage}}

\begin{document}

\title{Instance optimality of the adaptive maximum strategy}

\author[L.~Diening]{Lars Diening}
\address{Lars Diening,
 Mathematisches Institut der Universit\"at M\"unchen,
 Theresienstrasse 39, D-80333 M\"unchen, Germany
 }%
\urladdr{http://www.mathematik.uni-muenchen.de/~diening}
\email{diening@math.lmu.de}

\author[C.~Kreuzer]{Christian Kreuzer}
\address{Christian Kreuzer,
 Fakult\"at f\"ur Mathematik,
 Ruhr-Universit\"at Bochum,
 Universit\"atsstrasse 150, D-44801 Bochum, Germany
 }%
\urladdr{http://www.ruhr-uni-bochum.de/ffm/Lehrstuehle/Kreuzer/index.html}
\email{christan.kreuzer@rub.de}

\author[R.P.~Stevenson]{Rob Stevenson}
\address{Rob Stevenson,
  Korteweg-de Vries Institute (KdVI) for Mathematics,
  University of Amsterdam,
  Science Park 904,
  1098 XH Amsterdam, The Netherlands
 }%
\urladdr{https://staff.fnwi.uva.nl/r.p.stevenson/}
\email{r.p.stevenson@uva.nl}




\begin{abstract}
In this paper, we prove that the standard adaptive finite element
method with a (modified) {\em maximum marking strategy} is
{\em instance optimal} for the {\em total error}, being the \new{square root} of the
\new{squared} energy error \new{plus} the
\new{squared} oscillation. This result will be derived in the model setting of
Poisson's equation on a polygon, linear 
finite elements, and conforming triangulations created by newest vertex
bisection.
\end{abstract}

\date{\small \today}

\maketitle



\section{Introduction}
\label{sec:prelim}
Adaptive algorithms for the solution of PDEs that have been proposed
since the 70's are nowadays standard tools in science and engineering.  
In contrast to uniform refinements, adaptive mesh modifications do not
guarantee that the maximal mesh size tends to zero. For
this reason, 
even convergence of adaptive finite element methods (AFEM's) was
unclear for a long time, though practical experiences often showed
optimal convergence rates. 

In one dimension, convergence of an AFEM for elliptic problems
was proved by Babu{\v{s}}ka and Vogelius in
\cite{BabuskaVogelius:84} under some heuristic assumptions.
Later, D{\"o}rfler introduced in \cite{Doerfler:96} a bulk chasing marking
strategy thereby proving linear convergence of an AFEM in two space
dimensions for a sufficiently fine initial triangulation. This
restriction was removed in \cite{MoNoSi:00,MoNoSi:02} by Morin, Nochetto, and
Siebert.

In \cite{BinDD04},  Binev, Dahmen and DeVore extended the AFEM
analysed in \cite{MoNoSi:00} by a so-called {\em coarsening}
routine, and showed that the resulting method is {\em instance
  optimal}, cf. also \cite{Binev:07}. This means that the
energy norm of the error in any approximation produced by the
algorithm, with underlying triangulation denoted as $\tria$, is less
than some constant multiple of the 
error 
w.r.t.  any \emph{admissible} triangulation $\tilde{\tria}$ satisfying $\# (\tilde{\tria}
\setminus \tria_\bot) \leq \lambda  
\# (\tria \setminus \tria_\bot)$, for some fixed constant $\lambda \in (0,1)$.
Here, an admissible triangulation is a
conforming triangulation, which is created by finitely many newest
vertex bisections  (NVB) from a
fixed initial triangulation, \new{which we denote as $\tria_\bot$}. 

In \cite{Ste07}, 
it was shown that already without the addition of coarsening, the AFEM
is {\em class optimal}: Whenever the solution can be approximated at
some asymptotic (algebraic) convergence rate $s$ by finite element
approximations, \new{and the right-hand side can be approximated by piecewise polynomials at rate $s$}, then 
the AFEM produces a sequence of approximations, which converges with
precisely this rate $s$. 
In \cite{CaKrNoSi:08},
a similar result was \new{presented} with a refinement routine that is not
required to produce ``interior nodes'', and with a different treatment
of the approximation of the right-hand side, that is assumed to be in $L_2$. In that paper, the AFEM
is considered as a procedure for reducing the {\em total error}, being \new{the square root of}
the \new{squared} error in the energy norm
\new{plus} the \new{squared} so-called oscillation. This is also the point of view
that will be taken in the present work. 

In the last few years,
in numerous works class optimality results for AFEMs have been derived
for arbitrary space dimensions, finite elements of arbitrary orders,
the error measured in $L^2$, right-hand sides in $H^{-1}$,
nonconforming triangulations, discontinuous Galerkin methods, general
diffusion tensors, (mildly) non-symmetric problems, nonlinear
diffusion equations, and indefinite problems.

In all these works 
the marking strategy is {\em bulk chasing}, also
called {\em D{\"o}rfler marking}. In 
\cite{MorinSiebertVeeser:08}, Morin, Siebert and Veeser considered 
also 
the maximum and equidistribution strategies, without
proving any rates though. 

In the present work, we consider a standard AFEM, so without
coarsening, in the model setting
of
Poisson's equations with homogeneous Dirichlet boundary conditions on a
two-dimensional polygonal domain, the error
measured in the energy norm, square integrable right-hand side, linear
finite elements,
and conforming triangulations created by NVB.
The refinement routine is not required to create
interior nodes in refined triangles. 
Our method utilizes a (modified) {\em maximum
  marking strategy} for the standard residual error
estimator organised by edges. 

The maximum strategy marks all edges for bisection whose
 indicator is greater or equal to a constant $\new{\sqrt{\mu}} \in (0,1]$ times the
largest  indicator. This strategy is usually preferred by practitioners
since, other than
with D\"{or}fler marking, it does not require the sorting of the error
indicators, and in practise 
the results turn out to be very insensible to the
choice of the marking parameter $\mu \in (0,1]$.

Roughly speaking, our modification of the maximum marking strategy
replaces the role of the  error indicator associated with an edge $S$ by
the \new{square root of the} sum of the \new{squared} error indicators over those edges that necessarily have
to be bisected together with $S$ in order to retain a conforming
triangulation. The precise AFEM is stated in Section~\ref{sec:afem}.

\new{The main result of this paper states that for any $\mu \in (0,1]$, for some constants $C,\tilde C \new{\geq 1}$ it holds that
  \begin{align*}
    \abs{u-\uT[\tria_k]}_{H^1(\Omega)}^2+
    \osc_{\tria_k}^2(\tria_k) \le \tilde C\, \big(
    \abs{u-\uT[\tria]}_{H^1(\Omega)}^2+\osc_\tria^2(\tria)
    \big),
  \end{align*}
  for all admissible $\tria$ with \new{$\#(\tria \setminus \tria_\bot)\le \frac{\#(\tria_k \setminus \tria_\bot)}{C}$}.
  Here $\uT[\tria]$ is the Galerkin approximation to the exact solution $u$ from the finite element space corresponding to $\tria$, 
  $\osc_\tria^2(\tria) := \sum_{T \in \tria} |T|
  \norm{f-f_T}_{L^2(T)}^2$,
where $f_T := \frac{1}{\abs{T}} \int_T f\,dx$,
  and 
  $\tria_k$ is the triangulation produced in the $k$th iteration of our AFEM.
  This result means that our AFEM is {\em instance optimal} for the {\em total error}.}
Clearly, instance optimality implies class optimality for any
(algebraic) rate $s$, but not vice versa.

\new{Our AFEM is driven by the usual residual based a posteriori
  estimator, that is only equivalent to the {\em total}
  error. Consequently, we do not obtain instance optimality for the
  plain energy error, so without the oscillation term. 
}
  The oscillation encodes approximability of the right hand side
  and is in most cases 
  of higher order (e.g. \new{as} when $f \in H^s$ \new{for some $s>0$}), and thus asymptotically
  neglectable. 
 
 To prove instance optimality, we will show that the {\em total energy} associated with any
triangulation $\tria$ produced by our AFEM is not
larger than the total energy associated with any conforming
triangulation $\tilde{\tria}$ created by NVB with $\# (\tilde{\tria}
\setminus \tria_\bot) \leq \lambda \# (\tria \setminus \tria_\bot)$,
for some fixed constant $\lambda\in(0,1)$. 
Here the total energy is defined as  the Dirichlet energy \new{plus} the ``\new{squared} element residual part of 
the a posteriori estimator''.

The outline of this paper is as follows:
Sect.~\ref{sec:newest-bisection} is devoted to the newest vertex
bisection refinement procedure.  
On the set of vertices of the triangulations that can be created by
NVB from $\tria_\bot$, we introduce 
a tree structure where nodes generally have multiple parents.
Because of the resemblance of this tree structure with that
of a family tree, we refer to such a 
tree as a population. The concept of population is the key for the derivation of
some interesting new properties of  NVB. 

In Sect.~\ref{sec:projection}, we show that the squared norm of the
difference of  Galerkin solutions on nested triangulations
is equivalent to the sum of squared norms of the
differences of the Galerkin solution on the fine triangulation and that on some intermediate
triangulations. We call this the \emph{lower diamond estimate}.

Sect.~\ref{sec:error-bounds} is devoted to a posteriori error
bounds. It is shown that the difference of total energies associated
with  a triangulation $\tria_*$ and a coarser triangulation $\tria$ is
equivalent to the sum of the \new{squared} error indicators over exactly those edges
in $\tria$ that are refined in $\tria_*$. 

Based on the presented refinement framework and error estimator, 
we precisely specify our AFEM in Section~\ref{sec:afem}.

In Sect.~\ref{sec:fine}, we investigate some fine properties of
populations, and thus of conforming triangulations created by NVB.  
Calling the vertices in such a triangulation ``free'' when they can be
removed while retaining a conforming triangulation, the most striking
property says that  
the number of free nodes cannot be reduced by more than a constant
factor in any further conforming NVB refinement. 

Finally, in Sect.~\ref{sec:simpl-optim} we combine these tools to prove
instance optimality of our AFEM.\\ 

Throughout this paper we use the notation $a\Cleq b$ to indicate
$a\leq C\,b$, with a generic constant $C$ only \new{possibly} depending on fixed
quantities like the initial triangulation $\tria_\bot$, which will be
introduced in the next subsection.  Obviously, $a\Cgeq b$ means
$b\Cleq a$, and we denote $a \Cleq b\Cleq a$ by $a\eqsim
b$.

\section{Newest vertex bisection}
\label{sec:newest-bisection}
We recall properties of the newest vertex bisection (NVB) algorithm
for creating locally refined triangulations. Moreover, we introduce new concepts
related to conforming NVB that allow us to derive some new
interesting properties.
\subsection{Triangulations and binary trees}
\label{sec:refinement}

We denote by $\tria_\bot$ a conforming initial or ``bottom''
triangulation of a polygonal domain $\Omega \subset \mathbb{R}^2$.
We restrict ourselves to mesh adaptation 
by newest vertex bisection in $2d$; compare with
\cite{Baensch:91,Kossaczky:94,Maubach:95,Traxler:97,BinDD04,Ste08} as
well as \cite{NoSiVe:09,SchmidtSiebert:05} and the references therein.

To be more precise, for each $T \in \tria_\bot$, we label one of its
vertices as its {\em newest vertex}. Newest vertex bisection splits 
$T$ into two sub-triangles by
connecting the newest vertex to the midpoint of the opposite edge of
$T$, called the {\em refinement edge} of $T$. This midpoint is
labelled as the newest vertex of both newly created triangles, called
{\em children} of $T$.  A recursive application of this
rule uniquely determines all possible
NVB refinements of $\tria_\bot$. 

The triangles of any triangulation of $\Omega$ that can be
created in this way are the leaves of a subtree of an infinite binary
tree $\mathfrak{T}$ of triangles, having as roots the triangles of
$\tria_\bot$.  The newest vertex of any $T \in \mathfrak{T}$ is
determined by the labelling of newest vertices in $\tria_\bot$.  We
define the generation $\generation(T)$ of $T \in \mathfrak{T}$ as the
number of bisections that  are needed to create $T$ starting from
$\tria_\bot$. In particular, $\generation(T)=0$ for $T \in
\tria_\bot$. We have {\em uniform shape regularity} of $\mathfrak{T}$ 
in the sense that
$$
\sup_{T \in \mathfrak{T}} {\rm
  diam}(T)/|T|^{\frac{1}{2}}<\infty,
$$
only dependent on $\tria_\bot$. 
\new{We denote by $\nodes(T)$ the set of {\em nodes} or vertices of $T \in \mathfrak{T}$.}

Among all
triangulations that can be created by newest vertex bisection from
$\tria_\bot$, we \new{are interested in}
 those that are {\em conforming}
and denote the set of these triangulations as $\bbT$. Note that
$\tria_\bot \in \bbT$ by assumption.

In the following we shall {\em always} assume that in $\tria_\bot$ the
labelling of the newest vertices is such that
$\new{\tria=\tria_\bot}$ satisfies the {\em matching \new{condition}}:
\begin{gather}
  \label{ass:matching}
  \text{\parbox{10cm}{\em If, for $T,T' \in \tria$, $T \cap T'$
      is
      the refinement edge of $T$, \\[1mm]
      then it is the refinement edge of $T'$.}}
\end{gather} 
It is shown in \cite{BinDD04}, that 
such a labelling can be found for any conforming $\tria_\bot$.

\new{By induction one shows that for any $k \in \mathbb{N}_0$, the uniform refinement of the initial triangulation $\{T\in\mathfrak{T}\colon \generation(T)=k\}$ is in $\bbT$, and satisfies the matching condition. Moreover,}
the following result is valid:
\begin{proposition}[{\cite[Corollary 4.6]{Ste08}}]
  \label{pro:matching}
  Let $\tria \in \bbT$ and $T,T' \in \tria$ be such that $S=T \cap T'$
  is the refinement edge of $T$. Then,
  \begin{itemize}
  \item either
    $\generation(T')=\generation(T)$ and 
    $S$ is the refinement edge of $T'$, or
  \item
    $\generation(T')=\generation(T)-1$
    and $S$ is the refinement edge of one of the two children of $T'$.
  \end{itemize}
\end{proposition}

We denote by $\sides(\tria)$ ($\sides_0(\tria)$) the set of
  (interior) {\em sides} or edges, and by ${\nodes(\tria)}$ 
  (${\nodes_0(\tria)}$) the set of (interior) {\em nodes} or vertices
  of a triangulation $\tria \in \bbT$.

Finally, we note that if, for $\tria\in \bbT$, we
replace each $T \in \tria$ by its grandchildren, i.e., the children of
its children, then we obtain a conforming triangulation, that will be
denoted as $\tria\plus$; compare with Figure~\ref{fig:tria++}.

\begin{figure}[h]
  \centering
  \resizebox{0.43\textwidth}{!}{\input{tria.pspdftex}}
  \hspace{1cm}
  \resizebox{0.43\textwidth}{!}{\input{tria2.pspdftex}}
  \caption{$\tria$ with dashed refinement edges and the resulting
    $\tria\plus$.}\label{fig:tria++} 
\end{figure}

\subsection{Populations} 
\label{sec:populations}

\new{A} triangulation $\tria \in \bbT$ can alternatively be described in
terms of {\em populations}, which we shall introduce now.  To this
end, we denote the elements of 
$$
\new{\Pop^\top:=\bigcup_{T \in \mathfrak{T}} \nodes(T)},
$$
i.e., \new{the union of the vertices of all $T \in \mathfrak{T}$}, as {\em
  persons}.

We call a collection of persons a {\em population} when it is
equal to $\nodes(\tria)$ for some $\tria \in \bbT$, and denote with
$\bbP$ the collection of all populations. Since a 
triangulation $\tria \in \new{\bbT}$ is uniquely defined by its nodes
$\nodes(\tria)$, we have a one-to-one correspondence
between populations $\Pop \in \bbP$ and triangulations $\tria \in
\bbT$. 
When $\Pop = \nodes(\tria)$, we write $\Pop=\Pop(\tria)$ respectively 
$\tria=\tria(\Pop)$ and set 
$$
\Pop_\bot:=\Pop(\tria_\bot)
$$
for the initial or bottom population.

\new{The set $\Pop^\top$ can be equipped with a family tree structure.}
Let $P\in\Pop^\top\setminus\Pop_\bot$, then there exists a 
$T\in\mathfrak{T}$ such that $P$ is the midpoint of the refinement edge of
$T$. We call the newest vertex of $T$ a parent of $P$, respectively $P$
its child. If $P\in\partial\Omega$ then $P$ has one parent. Otherwise,
when $P\in\Omega$, it has two parents.  
The \emph{generation} of $P$ is defined by
$\generation(P)=\generation(T)+1$. 
\new{Since any uniform refinement of the initial partition is conforming and satisfies the matching condition, this definition is unique.}
Indeed,  
if $P$ is the midpoint of a
refinement edge of another
element in $\tilde T\in\mathfrak{T}$, then
$\generation(T)=\generation(\tilde T)$. 
Defining $\generation(P)=0$ when $P
\in \Pop_\bot$, we infer that the generation of a child is one plus
the generation of its parent(s), which in particular are of equal
generation.

\new{Since an equivalent definition of $\generation(P)$ is given by $\min\{\generation(T):T\in\mathfrak{T},\,P \in \nodes(T)\}$, no two vertices of a $T \in \mathfrak{T}$ can have the same generation, unless they have generation zero.}

Thanks to the uniform shape regularity of $\mathfrak{T}$, the number of children
a single person can have is uniformly bounded.  It is easy to see,
that a person $P\in\Pop^\top\setminus \Pop_\bot$ 
\new{has either two} \new{(when it is on the boundary)} \new{or four}
children \new{in $\Pop^\top$}; cf. Figure~\ref{parent-child}.
\begin{figure}
\input{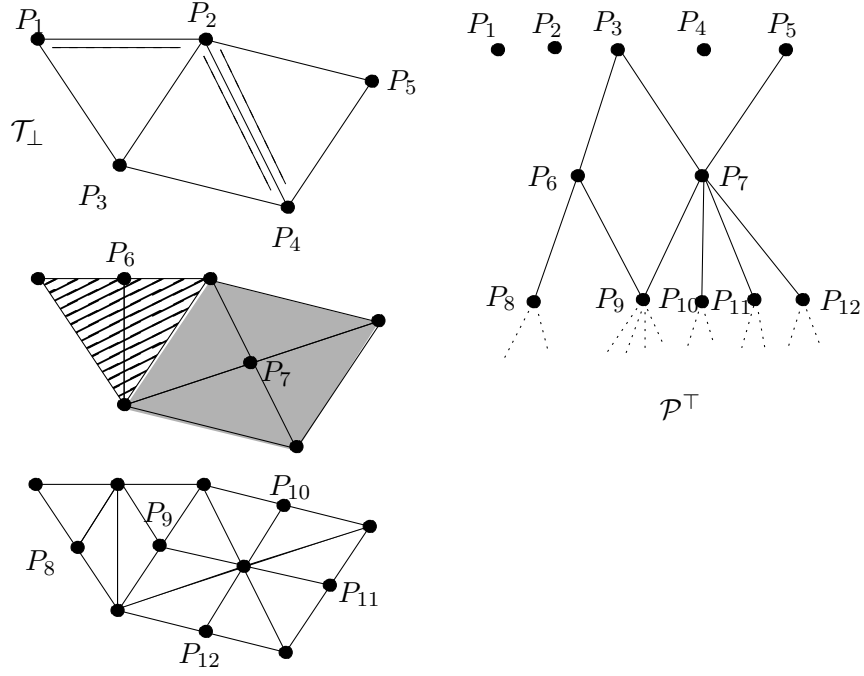}
\caption{Parents--children relations, and $\Omega(P_6)$ and $\Omega(P_7)$.}
\label{parent-child}
\end{figure}
For $P\in \Pop^\top$, we denote by $\children(P)$  the
set of the children of~$P$, and by $\parents(P)$ the set of its parents.

\new{
Any $\tria \in \bbT$ is obtained from $\tria_\bot$ by a sequence of
simultaneous bisections of pairs of triangles that share their
refinement edge (or by individual bisections of triangles that have
their refinement edge on the boundary). Each of such (simultaneous)
bisections corresponds to the addition of a person to the population
whose both its parents (or its single parent when the person is on the
boundary) are already in the population. 
We conclude the following result. 
}

\begin{proposition} 
  \label{pro:Pop_in_bbP}
  A collection $\mathcal{U} \subset \Pop^\top$ is a population if and
  only if $\Pop_\bot\subset \mathcal{U}$ and, for each $P \in
  \mathcal{U}$, we have that {\em all} parents of $P$
    are contained in $\mathcal{U}$. 
\end{proposition}

\new{This intrinsic characterization of a population as a family tree will be the key to derive many interesting properties of populations, and so of triangulations in $\bbT$.}

As we have seen above, a person $P \in \Pop^\top \setminus \Pop_\bot$ is the
(newest) vertex of four, or, when $P \in \partial\Omega$, two
triangles from $\mathfrak{T}$, each of them having the same generation as $P$.
For $P\in\Pop^\top$, we set
\begin{align*}
  \Omega(P):=\operatorname{interior}\bigcup\big\{T\in \mathfrak{T}\colon
  P \in T~\text{and}~\generation(T)=\generation(P)\big\},
\end{align*}
\new{cf. Figure~\ref{parent-child}.} This definition extends to subsets $\mathcal{U} \subset \Pop^\top$  setting
\begin{align}\label{eq:OmegaPop}
  \Omega(\mathcal{U}) &:= \operatorname{interior} \bigcup_{P \in
    \mathcal{U}} \overline{\Omega(P)}.
\end{align}

One easily verifies the following result:
\begin{proposition} \label{prop:areas}\mbox{}
  \begin{enumerate}
  \item \label{areas_item1} Let $P_1, P_2 \in \Pop^\top\setminus
    \Pop_\bot$ with $P_1 \neq P_2$ and
    $\generation(P_1)=\generation(P_2)$.  Then $\Omega(P_1) \cap
    \Omega(P_2)=\emptyset$.  \item \label{areas_item3} Let $P \in \Pop^\top \setminus \Pop_\bot$.
    Then $\Omega(P) \subset \Omega(\parents(P))$.
 \end{enumerate}
\end{proposition}

\subsection{Refinements and coarsenings}
\label{ssec:refine}

For $\tria,\tria_* \in \bbT$ we write $\tria \leq \tria_*$ or $\tria_*
\geq \tria$, when  $\tria_*$ is a {\em refinement} of $\tria$ or,
equivalently, $\tria$ is a {\em coarsening} of $\tria_*$, i.e., when the
tree of $\tria$ is a subtree of that of $\tria_*$.  This defines a
partial ordering on~$\bbT$.  On $\bbP$, we define a {\em partial
  ordering} by $\Pop \le \Pop_*$ when $\Pop \subset \Pop_*$. We call
$\Pop_*$ a {\em refinement} of~$\Pop$ or, equivalently, $\Pop$ a {\em
  coarsening} of~$\Pop_*$. These orderings are equivalent:
\begin{proposition}\label{P:ordering}
For $\Pop,\Pop_* \in \bbP$, we have
\begin{align*}
  \Pop\le\Pop_* \quad\Longleftrightarrow
  \quad\tria(\Pop)\le\tria(\Pop_*).
\end{align*}
\end{proposition}
The partially ordered set $(\bbP,\le)$ is a {\em lattice}, since for
any $\Pop_1,\Pop_2 \in \bbP$, the {\em lowest upper bound} $\Pop_1
\vee \Pop_2$ and the {\em greatest lower bound} $\Pop_1 \wedge \Pop_2$
exist in $\bbP$, and are given by
\begin{equation} 
  \label{eq:wedgeveepop} 
  \Pop_1 \vee \Pop_2=\Pop_1
  \cup \Pop_2\quad\text{and}\quad \Pop_1 \wedge \Pop_2=\Pop_1 \cap \Pop_2,
\end{equation}
respectively. We call $\Pop_1 \wedge \Pop_2$ the {\em largest common
  coarsening}, and $\Pop_1 \vee \Pop_2$ the {\em smallest common
  refinement} of $\Pop_1$ and $\Pop_2$.

Since $\Pop_\bot \leq \Pop$ for all $\Pop \in \bbP$, we have that 
$\Pop_\bot$ is the {\em bottom} of $(\bbP,\le)$.  Moreover, if we
define $\bbPhat := 
\bbP \cup \set{\Pop^\top}$ and set $\Pop^\top \geq\Pop$ for all $\Pop
\in \bbPhat$, then $\Pop^\top$ is the {\em top} of~$\bbPhat$ and
whence
$\bbPhat$ is a bounded lattice. 

These notions can be transferred to
triangulations $\tria_1, \tria_2
\in \bbT$ via
\begin{align*}
  \tria_1 \vee \tria_2 &:= \tria\big( \Pop(\tria_1) \vee \Pop(\tria_2)
  \big),
  \\
  \tria_1 \wedge \tria_2 &:= \tria\big( \Pop(\tria_1) \wedge
  \Pop(\tria_2) \big).
\end{align*}
Consequently, $(\bbT, \le)$ is a lattice with
bottom~$\tria_\bot$. 
Moreover, we can
add a largest element $\tria^\top= \tria(\Pop^\top)$ to $\bbT$
and define $\bbThat := \bbT \cup \set{\tria^\top}$ and $\tria^\top
\geq \tria$ for all $\tria \in \bbT$. Then $\tria^\top$ is the top
of the bounded lattice $\bbThat$.
\begin{remark}
  \label{rem:interpret}
  An interpretation of $\tria_1 \vee \tria_2$ and $\tria_1 \wedge
  \tria_2$ is given in the following (cf.  \cite[Lemma
  4.3]{NoSiVe:09}).  For $\tria_1,\tria_2 \in \bbT$, let $T_1 \in
  \tria_1$, $T_2 \in \tria_2$ with $\new{|T_1 \cap T_2|}>0$, so
  that either $T_1 \subset T_2$ or $T_2 \subset T_1$.  W.l.o.g., we assume $T_1 \subset T_2$. Then $T_1 \in \tria_1
  \vee \tria_2$ and $T_2 \in \tria_1 \wedge \tria_2$.
\end{remark}

For $\tria \in \bbT$ and $\mathcal{U} \subset \tria$, we define
\begin{align*}
  \Omega(\mathcal{U}) &:= \operatorname{interior} \bigcup \set{T \,:\,
    T \in \mathcal{U}}.
\end{align*}
For $\tria,
\tria_* \in \bbT$ with $\tria \leq \tria_*$, we call $\Omega(\tria
\setminus \tria_*)=\Omega(\tria_* \setminus \tria)$ the {\em area of
  coarsening}. It
is the union of all triangles that are coarsened when passing
from~$\tria_*$ to $\tria$, or, equivalently, the union all triangles that are refined when passing
from~$\tria$ to $\tria_*$. The coarsening point of view, however, will often turn out to be more relevant, in particular in Sect.~\ref{sec:projection}.

Recalling the definition $\tria\plus$ for $\tria \in \bbT$, we set $\Pop\plus
:= \Pop((\tria(\Pop)\plus)$. 
Then \new{$\Pop\plus \setminus \Pop \subset \big(\cup_{P \in \Pop} \children(P) \cup \children(\children(P))\big) \setminus \Pop$, with equality only when all $T \in \tria(\Pop)$ have the same generation, cf. Figure~\ref{P++}.}
\begin{figure}
\input{P++.pspdftex}
\caption{$\Pop$,  $\Pop\plus$, and grandchildren  ($\Box$)  of $P \in \Pop$ that are not in $\Pop\plus$}
\label{P++}
\end{figure}
There is a one-to-one correspondence of $\sides(\tria(\Pop))$ and
$\Pop\plus \setminus \Pop$. Indeed,  denote the midpoint
of a side~$S \in
  \sides(\tria)$ by
$\midpoint{S}$ and set
$\midpoint{\mathcal S}:=\{\midpoint{S}:S \in \sides\}$ for a collection~$\sides$ of sides, then
\begin{align}
  \label{eq:midpoints_pre}
  \Pop\plus \setminus \Pop = \midpoint{\sides(\tria(\Pop))}.
\end{align}
More general, if $\Pop, \Pop_* \in \bbP$ with $\Pop \leq \Pop_*$, then
\begin{equation} 
  \label{eq:midpoints} 
  \Pop_* \cap (
  \Pop\plus \setminus \Pop) =\midpoint{\sides(\tria(\Pop))
    \setminus \sides(\tria(\Pop_*))}.
\end{equation}

For $\tria \in \bbT$, we define 
\begin{align*}
\VoT[\tria] &:=\set{v \in H_0^1(\Omega):v|_T \in P_1(T)\,(T \in \tria)},\\
\mathbb{V}(\tria) &:=\set{v \in H^1(\Omega):v|_T \in P_1(T)\,(T \in \tria)},
\end{align*}
Thanks to the nodal Lagrange basis representation of any finite element function,
the degrees of freedom (DOFs) of $\VoT$ or $\mathbb{V}(\tria)$ can
  be identified with ${\nodes_0(\tria)}$ or ${\nodes(\tria)}$, respectively.
  We set $\VoT[\tria^\top] := H^1_0(\Omega)$ and
$\mathbb{V}(\tria^\top) := H^1(\Omega)$.

The proof of the next proposition is left to the reader.

\begin{proposition}\label{prop:V-lattice}
  The mapping $\tria \mapsto \VoT[\tria]$ from $\bbThat$ to the
  lattice of vector spaces is compatible with the lattice
  structure, i.e.,
  \begin{align*}
    \tria\le\tria_* \quad&\Rightarrow
    \quad\VoT[\tria] \subset \VoT[\tria_*],
    \\
    \VoT[\tria \wedge \tria_*] &=  \VoT[\tria] \cap \VoT[\tria_*],
    \\
    \VoT[\tria \vee \tria_*] &=  \VoT[\tria] + \VoT[\tria_*],
  \end{align*}
  The same holds true when we replace $\VoT$ by $\mathbb{V}(\tria)$.
\end{proposition}

\subsection{The refinement routine}

For $\Pop \in \bbP$ and a finite set $\mathcal{C} \subset \Pop^\top$,
we denote by $\Pop \oplus \mathcal{C}$ the {\em smallest refinement of
  $\Pop$ in $\bbP$ that contains~$\mathcal{C}$}, i.e.,
\begin{align*}
  \Pop \oplus \mathcal{C} &:= \bigwedge \set{\Pop' \in \bbP\,:\, \Pop'
    \geq\Pop, \mathcal{C} \subset \Pop'}.
\end{align*}
This is well-defined. To see this, \new{recall that,  thanks to the
  matching condition, we have that}
 $\{P \in
\Pop^\top\,:\,\generation(P) \leq k\} \in \bbP$ \new{for all $k\in\setN_0$},
\new{so that the largest common coarsening can be taken over}
the finitely many $\Pop' \in \bbP$ with
$\max_{P \in \Pop'} \generation(P) \leq \max_{P \in \Pop \cup
  \mathcal{C}} \generation(P)$.

For $P \in \Pop^\top$, we also write
$\Pop \oplus P$ instead of  $\Pop \oplus \set{P}$.  Note that 
$$
\Pop_* \oplus \Pop = \Pop_* \oplus (\Pop \setminus
\Pop_*) = \Pop_* \vee \Pop \qquad \text{for all }\Pop,\Pop_*
\in \bbP.
$$

For $\Pop \in \bbP$ and $\mathcal{C} \subset \Pop^\top \setminus
\Pop_\bot$ we denote by $\Pop \ominus \mathcal{C}$ the {\em greatest
  coarsening of~$\Pop$ in $\bbP$ that does not contain~$\mathcal{C}$},
i.e.,
\begin{align*}
  \Pop \ominus \mathcal{C} &:= \bigvee \set{\Pop' \in \bbP \,:\,
    \Pop' \leq\Pop,\, \mathcal{C} \cap \Pop' = \emptyset}.
\end{align*}
For $P \in \Pop^\top \setminus \Pop_\bot$, we also write $\Pop \ominus
P$ for $\Pop \ominus \set{P}$.

For $\tria \in \bbT$ and $\mathcal{U} \subset \tria$, \new{we
  denote  by} 
$\tria_*=\refi{\tria}{\mathcal{U}}$  the {\em smallest
  refinement of $\tria$ in $\bbT$ with $\mathcal{U} \cap
  \tria_*=\emptyset$}, i.e.,
$$
\refi{\tria}{\mathcal{U}}\:= \bigwedge \set{\tria' \in \bbT\,:\, \tria'
  \geq\tria,\, \mathcal{U} \cap \tria'=\emptyset}.
$$
\new{In this definition $\tria'\in\bbT$ can be} \new{restricted to $\tria' \leq \tria\plus$.} 
The set $\mathcal{U}$ is commonly referred to as the set of triangles
that are marked for refinement.

Although no uniform bound for $\#(\tria_* \setminus \tria)
/\#\mathcal{U}$ can be shown, the following important result is valid:
\begin{theorem}[\protect\cite{BinDD04}] 
  \label{thm:BDD}
  For any sequence $(\tria_k)_k \subset \bbT$ defined by
  $\tria_0=\tria_\bot$ and
  $\tria_{k+1}=\refi{\tria_{k}}{\mathcal{U}_k}$ for some ${\mathcal
    U}_k \subset \tria_{k} $, $k=0,1,\ldots$, we have  that
  $$
  \# (\tria_k \setminus \tria_\bot) \Cleq  \sum_{i=0}^{k-1} \#{\mathcal U}_i.
  $$
\end{theorem}

Our adaptive finite element routine will be driven by the marking of
{\em edges} for refinement. Therefore, for $\tria \in \bbT$ and
$\mathcal{M} \subset \sides(\tria)$, let
$\tria_*=\refine{\tria}{\mathcal{M}}$ denote the {\em smallest
  refinement of $\tria$ in $\bbT$ with $\mathcal{M} \cap
  \sides(\tria_*)=\emptyset$}, i.e.,
$$
\refine{\tria}{\mathcal{M}} \:= \bigwedge \set{\tria' \in \bbT\,:\,
  \tria' \geq\tria,\, \mathcal{M} \cap \sides(\tria')=\emptyset}.
$$
Note that $\refine{\tria}{\mathcal{M}} =\tria(\Pop(\tria)\oplus \midpoint{\mathcal{M}})$.

Setting
\begin{align*}
  {\mathcal U}_1& =\{T \in \tria: \mathcal{M} \text{ contains an edge
    of } T\},
  \\
  {\mathcal U}_2& =\{T' \in \new{\refi{\tria}{\mathcal{U}_1}}: \mathcal{M}
  \text{ contains an edge of } T'\},
\end{align*}
we have that 
$$
\refine{\tria}{\mathcal{M}}=\refi{\refi{\tria}{\mathcal{U}_1}}{\mathcal{U}_2}.
$$
Since moreover $\#\mathcal{U}_1+\#\mathcal{U}_2 \leq 4 \cdot
\#\mathcal{M}$, we conclude the following result.

\begin{corollary} 
  \label{cor:BDD}
  For any sequence $(\tria_k)_k \subset
  \bbT$ defined by $\tria_0=\tria_\bot$ and
  $\tria_{k+1}=\refine{\tria_k}{\mathcal{M}_k}$ for some ${\mathcal
    M}_k \subset \sides(\tria_{k})$, $k=0,1,\ldots$, we have
 that
  $$
  \# (\tria_k \setminus \tria_\bot) \Cleq \sum_{i=0}^{k-1} \#{\mathcal
    M}_i .
  $$
\end{corollary}

\new{
Since every simultaneous bisection of a pair of triangles that share their refinement edge increases the population by one, and
the number of triangles by two, and every bisection of a triangle that has its refinement edge on the boundary increases both the population and the number of triangles by one,} we observe that for
$\Pop, \Pop_* \in \bbP$ with $\Pop_* \geq \Pop$,
\begin{align}
  \label{eq:no_P_vs_T}
  \# (\Pop_* \setminus \Pop) &\leq \# (\tria(\Pop_*) \setminus
  \tria(\Pop)) \leq 2\, \# (\Pop_* \setminus \Pop).
\end{align}
This result will allow us to transfer Corollary~\ref{cor:BDD} in terms of
populations.

\section{Continuous problem, its discretisation, \\and the lower diamond
  estimate} \label{sec:projection} 
In this section, we shall introduce the model problem. Moreover, we
shall investigate a splitting of the difference of energies related to
nested spaces. To \new{the best of our knowledge}, this so-called lower diamond estimate  
is new, and it plays a crucial role in the proof of the instance
optimality of the AFEM in Section~\ref{sec:simpl-optim}.
\subsection{Continuous and discrete problem}
We consider the model setting of  Poisson's equation 
\begin{alignat}{2}
  \label{eq:problem}
  \begin{aligned}
    -\Delta u &= f &&\qquad \text{on } \Omega,
    \\
    u &=0 &&\qquad \text{on } \partial \Omega,
  \end{aligned}
\end{alignat}
\new{where, in view of the application of an a posteriori error estimator, we assume that $f \in L^2(\Omega)$.}
In weak form, \new{it} reads as finding $u:=u_{\tria^\top} \in H^1_0(\Omega)\new{=\mathbb{V}_0(\tria^\top)}$ such that
$$
\int_\Omega \nabla u \cdot \nabla v \,\dx=\int_\Omega fv\,\dx \qquad(v \in H^1_0(\Omega)).
$$

For $\tria \in \bbT$, the Galerkin approximation $\uT\in\VoT$ of $u$
is uniquely defined by
\begin{align} \label{eq:discrete}
  \int_\Omega \nabla \uT \cdot \nabla \vT \,\dx = \int_\Omega f\vT\,\dx
  \qquad(\vT \in \VoT).
\end{align}

It is well known, that for \new{$\tria \in \hat{\bbT}$, $\uT$ 
is the unique minimiser of
the (Dirichlet) energy
\begin{align*}
  \mathcal{J}(v) := \int_\Omega \frac 12 \abs{\nabla v}^2 -\new{f v}\,\dx 
   \qquad (v\in \VoT).
\end{align*}
}

Setting
$$
   \mathcal{J}(\tria) := \mathcal{J}(\uT),
$$
Proposition~\ref{prop:V-lattice}
shows that $\mathcal{J}$ is non-increasing with respect to~$(\bbThat,\le)$, i.e.,
for $\tria,\tria_* \in \bbThat$, we have
\begin{align}
  \label{eq:J-monotone}
  \tria \leq \tria_* \qquad \Rightarrow \qquad \mathcal{J}(\tria) \geq
  \mathcal{J}(\tria_*).
\end{align}
\new{Moreover, from basic calculations we observe that for $\tria \in
\bbT$, we have
\begin{equation} \label{eq:Ediff}
\mathcal{J}(\tria) - \mathcal{J}(\tria_*) =
 \frac12
\abs{\uT-\uT[\tria_*]}^2_{H^1(\Omega)}
\end{equation}
for all
$\tria_*\in \hat{\bbT}$ with $\tria\le\tria_*$.}

\subsection{The lower diamond estimate}
To formulate the main result from this subsection, we have to start with a definition.
\begin{definition} \label{def:lower_dia_tria} For
    $\set{\tria_1, \dots, \tria_m} \subset \bbT$, we call
    $(\tria^\wedge,\tria_\vee;\tria_1, \dots, \tria_m)$ a {\em lower
      diamond} in $\bbT$ of size $m$, 
       when $\tria^\wedge=\bigwedge_{j=1}^m \tria_j$,
    $\tria_\vee=\bigvee_{j=1}^m \tria_j$, and the areas of {\em
      coarsening} $\Omega(\tria_j \setminus \tria_\vee)$ are pairwise
    disjoint, cf. Figure~\ref{fig:diamond} for an illustration.
  
  It is called an {\em upper diamond} in $\bbT$ of size $m$, when the last condition reads as
  the areas of {\em refinement} $\Omega(\tria^\wedge \setminus \tria_j)$ being
  pairwise disjoint.
  \end{definition}

\begin{figure}[h]
\centering
\input{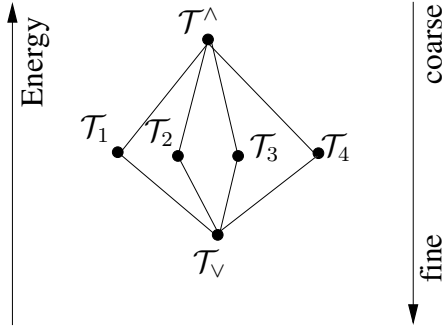}  
\caption{Lower (or upper) diamond of size~4.}
 \label{fig:diamond}
\end{figure}

Obviously, for any $\tria \in \bbT$, $(\tria,\tria;\tria)$ is a lower
(and upper) diamond in $\bbT$ of size 1. More interesting is the
following result: 
\begin{lemma} \label{lem:lowersize2} 
 For any $\tria_1 \neq \tria_2 \in \bbT$, $(\tria_1\wedge
    \tria_2,\tria_1 \vee \tria_2;\tria_1,\tria_2)$ is a lower (and
    upper) diamond in $\bbT$ of size 2.
\end{lemma}

\begin{proof}
Setting $\tria^\wedge:=\tria_1\wedge \tria_2$, $\tria_\vee:=\tria_1 \vee \tria_2$,
assume that $\Omega(\tria_1 \setminus
  \tria_\vee)$ and $\Omega(\tria_2 \setminus
  \tria_\vee)$ are not
  disjoint. Recalling that $\Omega(\tria_j\setminus
  \tria_\vee)=\Omega(\tria_\vee\setminus \tria_j)$, then
  there exists a $T\in (\tria_\vee \setminus
  \tria_1)\cap (\tria_\vee \setminus \tria_2)=\tria_\vee
  \setminus(\tria_1\cup\tria_2)$. This contradicts $\tria_\vee =
  \tria_1 \vee \tria_2$; compare also with Remark~\ref{rem:interpret},
  and thus $(\tria^\wedge,\tria_\vee;\tria_1,\tria_2)$ is a lower
  diamond. 
  
  Similarly, one finds that $(\tria_1 \setminus \tria^\wedge) \cap
  (\tria_2 \setminus \tria^\wedge)=\emptyset$, i.e.,
  $(\tria^\wedge,\tria_\vee;\tria_1,\tria_2)$ is an upper diamond. 
\end{proof}

The main goal of this subsection is to prove the following
result: 
\begin{theorem} \label{thm:stable-splitting}
  Let $(\tria^\wedge,\tria_\vee;\tria_1, \dots, \tria_m)$ be a lower diamond in $\bbT$.
  Then
  \begin{align}
    \abs{u_{\tria_\vee}- u_{\tria^\wedge}}_{H^1(\Omega)}^2 &\eqsim \sum_{j=1}^m\!
    \abs{u_{\tria_\vee}-u_{\tria_j}}_{H^1(\Omega)}^2,
  \end{align}
 only dependent on $\tria_\bot$.
\end{theorem}

The first ingredient to prove this theorem is the following observation.
\begin{lemma} \label{lemma:Lebesgue}
Let $\tria, \tria_* \in \bbThat$ with $\tria \leq \tria_*$,
and let $\Pi:\VoT[\tria_*]\to\VoT[\tria_*]$ be a \new{linear} projector onto
$\VoT$ which is  $H^1(\Omega)$-bounded, uniformly in $\tria,
\tria_*$.  Then,
$$
\abs{u_{\tria_*}-u_\tria}_{H^1(\Omega)} \eqsim \abs{u_{\tria_*}-\Pi
    u_{\tria_*}}_{H^1(\Omega)}.
    $$
\end{lemma}

\begin{proof} Use that $u_\tria$ is the best approximation from $\VoT$
  to $u_{\tria_*}$ in $|\cdot|_{H^1(\Omega)}$, and 
$|u_{\tria_*}-\Pi u_{\tria_*}|_{H^1(\Omega)} \leq |u_{\tria_*}-v_{\tria}|_{H^1(\Omega)} +|\Pi(v_{\tria}- u_{\tria_*})|_{H^1(\Omega)} 
\Cleq
\abs{u_{\tria_*} - v_{\tria}}_{H^1(\Omega)}$ for all $v_\tria
  \in \VoT$.
\end{proof}

In order to localize the projection error to the area of coarsening, we shall consider 
a particular Scott-Zhang type quasi-interpolator~\cite{ScoZha90}. 

\begin{lemma}
  \label{pro:proj} 
  Let $\tria, \tria_* \in \bbT$ with $\tria \leq \tria_*$. Let
  $\Omega_1 := \Omega(\tria \setminus \tria_*)$ and $\Omega_2 :=
  \Omega \setminus \overline{\Omega}_1$.  There exists a projector
  $\Pi_{\tria_*\to \tria}:H^1(\Omega) \rightarrow H^1(\Omega)$ onto
  $\VT$ with the following properties
  \begin{enumerate}[label={{\tt(Pr\arabic{*})}}]
  \item \label{itm:proj_bnd} $\abs{ \Pi_{\tria_* \to \tria}v}_{H^1(\Omega)} \lesssim \abs{v}_{H^1(\Omega)}$ for all $v \in H^1(\Omega)$. 
  \item \label{itm:proj_Omega_j} There exist projectors
    $\new{\bar{\Pi}_{\tria,i}}:H^1(\Omega_i)\rightarrow H^1(\Omega_i)$
    onto $\VT|_{\Omega_i}:=\{v|_{\Omega_i}:v \in \VT\}$, 
    \new{with for any $T \in \tria$ with $T \subset \bar{\Omega}_i$,
    $$
    |\bar{\Pi}_{\tria,i}
    v_i|^2_{H^1(T)}\Cleq \sum_{\{T' \in \tria: T' \cap T \neq \emptyset,\,T' \subset  \bar{\Omega}_i\}} |v_i|^2_{H^1(T')}$$}
    for all $v_i \in H^1(\Omega_i)$, $i=1,2$,
    such that
    \begin{align*}
      (\Pi_{\tria_*\to\tria} v)|_{\Omega_i}=\new{\bar{\Pi}_{\tria,i}}
      (v|_{\Omega_i}),\qquad i=1,2.
    \end{align*}
  \item \label{itm:proj_compl} $v - \Pi_{\tria_* \to \tria} v$
    vanishes on $\overline{\Omega}_2$ for all $v \in \VT[\tria_*]$.
  \item \label{itm:proj_V0} $\Pi_{\tria_* \to \tria}(
    \VoT[\tria_*]) \subset  \VoT$.
  \end{enumerate}
\end{lemma}

\begin{proof}
  For the construction of $\Pi_{\tria_*
    \to \tria}$, we assign to each node $z \in
  \nodes(\tria)$ 
  some edge $S_z \in \sides(\tria)$ such that \new{$z \in S_z$ and}
  \begin{align}\label{eq:SZSz}
    \begin{split}
       S_z \subset \begin{cases}
       \partial
        \Omega_1, \quad&\text{if}~z\in\partial\Omega_1,
        \\
        \partial
        \Omega_2, \quad&\text{if}~z\in\partial\Omega_2.
      \end{cases}
    \end{split}
  \end{align}
   These restrictions are well posed, since $\Omega$ is a domain, which
  excludes the case that $\Omega_1$ and $\Omega_2$ touch at some
  isolated point.  We denote by $\Pi := \Pi_{\tria_* \to \tria}$ the
  Scott-Zhang projector according to the above
  assignments~\eqref{eq:SZSz}, i.e., for $z\in \nodes(\tria)$, the
  nodal value $(\Pi v)(z)$ is defined by means of $L^2(S_z)$ dual
  functions of the local nodal basis functions on $S_z$; compare with
  \cite{ScoZha90}.  Then $\Pi \colon H^1(\Omega) \to H^1(\Omega)$ is a
  projector onto $\VT$, and~{\ref{itm:proj_bnd}} follows
  from~\cite{ScoZha90}.

  Thanks to~\eqref{eq:SZSz}, we may define the Scott-Zhang projectors
  $\new{\bar{\Pi}_{\tria,i}}:H^1(\Omega_i)\to H^1(\Omega_i)$ onto $\VT|_{\Omega_i}$
  according to $S_z$, $z\in\nodes(\tria)\cap\overline{\Omega}_i$,
  $i=1,2$.  \new{With these definitions the properties listed in {\ref{itm:proj_Omega_j}} are valid.}
    
  Let $v \in \VT[\tria_*]$. Then $v|_{\Omega_2} \in \VT|_{\Omega_2}$ and
  since $\new{\bar{\Pi}_{\tria,2}}$ is a projector onto $\VT|_{\Omega_2}$, we
  have that $\new{\bar{\Pi}_{\tria,2}} v|_{\Omega_2}=v|_{\Omega_2}$, i.e.,
  $\Pi v = v$ on
  $\overline{\Omega}_2$. This proves~{\ref{itm:proj_compl}}.
  
  In order to prove~{\ref{itm:proj_V0}} let $v \in \VoT[\tria_*]$.
  Then~{\ref{itm:proj_compl}} implies $v=0$ on $\partial \Omega \cap
  \partial \Omega_2$.  Therefore, let  $z \in \nodes(\tria)$ with $z \in
  \partial \Omega \setminus
  \partial \Omega_2$. Then locally $\partial\Omega_1$ coincides with
  $\partial\Omega$, and thus $S_z \subset \partial \Omega_1 \cap
  \partial \Omega$ according to~\eqref{eq:SZSz}. Since $v=0$ on
  $\partial \Omega$, it follows from properties of the Scott-Zhang
  projector that $(\Pi v)(z)= 0$. Consequently, we have $(\Pi
  v)(z)=0$ for all $z\in\nodes(\tria)$, i.e., $\Pi v = 0$ on $\partial
  \Omega$.
\end{proof}

\begin{remark}
  Note that the projector constructed in Lemma~\ref{pro:proj} does
  not map $H^1_0(\Omega)$ into $\VoT$ when $\Omega_1$ touches the
  boundary.  In such a situation we might have $z \in \nodes(\tria) \cap
  \partial \Omega \cap \partial \Omega_1 \cap
  \partial \Omega_2$ but $\partial \Omega \cap
  \partial \Omega_1 \cap
  \partial \Omega_2$ contains no edge. Hence, in view of  \eqref{eq:SZSz}
  it is not possible to
  require additionally that $z \in \partial\Omega$ implies $S_z
  \subset \partial\Omega$. 
\end{remark}
\begin{theorem}
  \label{thm:Pi_commutes}
  Let $(\tria^\wedge,\tria_\vee;\tria_1, \dots, \tria_m)$ be a lower diamond in $\bbT$.
  Set
  $\Pi_j := \Pi_{\tria_\vee \to
    \tria_j}$, and $\Omega_j := \Omega(\tria_j \setminus \tria_\vee)$,
   $j=1,\dots,m$. 
  Then the projectors $\Pi_j$
  commute as operators from $\VT[\tria_\vee] \to \VT[\tria_\vee]$.

  Define $\Pi := \Pi_1 \circ \dots \circ \Pi_m:\VT[\tria_\vee] \to \VT[\tria_\vee]$.
  Then $\Pi$ is a projector onto $\VT[\tria^\wedge]$, and $\Pi(\VoT[\tria_\vee]) \subset \VoT[\tria^\wedge]$.
  Moreover, for all $v_\vee
  \in \VT[\tria_\vee]$ we have
  \begin{align}
    \label{eq:Pi_c1}
    \Pi v_\vee &=
    \begin{cases}
      \Pi_j v_\vee &\qquad \text{on
        $\overline{\Omega}_j$},
      \\
      v_\vee &\qquad \text{on $\overline{\Omega \setminus \cup_{j=1}^m
          \Omega_j}$}
    \end{cases}
  \end{align}
  and $\abs{\Pi v_\vee}_{H^1(\Omega)} \Cleq
  \abs{v_\vee}_{H^1(\Omega)}$, only dependent on $\tria_\bot$.
\end{theorem}

\begin{proof}
  Thanks to Lemma~\ref{pro:proj} we have that $\Pi_j$ is a projector onto $\VT[\tria_j]$, and $\Pi_j(\VoT[\tria_\vee]) \subset
  \VoT[\tria_j]$.
  We fix $i\neq j$.   Since $\Omega_i$ and $\Omega_j$ are disjoint, we have
  $\overline{\Omega}_i \subset \overline{\Omega \setminus
    \overline{\Omega}_j}$. Hence, we conclude
  from~{\ref{itm:proj_compl}} that $\Pi_j v_\vee=v_\vee$ on
  $\overline{\Omega}_i$ and $\Pi_i v_\vee = \Pi_j\Pi_i v_\vee$ on
  $\overline{\Omega}_i$. 
  Since~{\ref{itm:proj_Omega_j}} implies that $(\Pi_i w)|_{\Omega_i}$ only depends on $w|_{\Omega_i}$, we conclude that 
  $\Pi_i \Pi_j v_\vee = \Pi_i v_\vee$ on $\overline{\Omega}_i$, and
  thus $\Pi_i \Pi_j v_\vee = \Pi_i v_\vee = \Pi_j \Pi_i
  v_\vee$ on $\overline{\Omega}_i$.  Analogously, we have 
  $\Pi_i \Pi_j v_\vee = \Pi_j v_\vee = \Pi_j \Pi_i v_\vee$ on
  $\overline{\Omega}_j$. Moreover, thanks to~{\ref{itm:proj_compl}}, we have
  $\Pi_i \Pi_j v_\vee = v_\vee = \Pi_j \Pi_i v_\vee$ on
  $\overline{\Omega \setminus (\Omega_i \cup \Omega_j)}$ and thus
  $\Pi_i \Pi_j v_\vee = \Pi_j \Pi_i v_\vee$. 

  Since the $\Pi_j$ commute, we conclude that $\Pi$ is a projector onto
  $\bigcap_{j=1}^m \VT[\tria_j] = \VT[\tria^\wedge]$; compare also with
  Proposition~\ref{prop:V-lattice}.  The claim $\Pi(
    \VoT[\tria_\vee]) \subset \VoT[\tria^\wedge]$ follows analogously
  using~{\ref{itm:proj_V0}}.
  Again since the $\Pi_j$ commute  we infer~{\eqref{eq:Pi_c1}} from
  {\ref{itm:proj_Omega_j}}. 

  Thanks to~\eqref{eq:Pi_c1} and~{\ref{itm:proj_Omega_j}} we
  conclude
  \begin{align*} 
    \abs{\Pi v_\vee}^2_{H^1(\Omega)} &= \abs{
      v_\vee}^2_{H^1(\Omega\setminus\bigcup_{j=1}^m\Omega_j)} +
    \sum_{j=1}^m \abs{\Pi_j v_\vee}^2_{H^1(\Omega_j)}
    \\
    &\Cleq \abs{
      v_\vee}^2_{H^1(\Omega\setminus\bigcup_{j=1}^m\Omega_j)} +
    \sum_{j=1}^m \abs{ v_\vee}^2_{H^1(\Omega_j)} =
    \abs{v_\vee}^2_{H^1(\Omega)},
  \end{align*}
  with constants independent of $m$.
\end{proof}

With the projectors $\Pi_j$ and $\Pi$ at hand, we are now ready to prove the main result of this section.
\begin{proof}[Proof of Theorem~\ref{thm:stable-splitting}]
Thanks to Lemma~\ref{pro:proj}, $\Pi, \Pi_j:\VoT[\tria_\vee]\to\VoT[\tria_\vee]$ are 
  (uniformly) bounded projectors onto $\VoT[\tria^\wedge]$ or
  $\VoT[\tria_j]$. From this, together with  Lemma~\ref{lemma:Lebesgue}, \eqref{eq:Pi_c1}, and
  {\ref{itm:proj_compl}}, we infer that
  \begin{align*}
    |u_{\tria_\vee}-u_{\tria^\wedge}|_{H^1(\Omega)}^2 & \eqsim
    |u_{\tria_\vee}-\Pi u_{\tria_\vee}|_{H^1(\Omega)}^2= \sum_{j=1}^m
    |u_{\tria_\vee}-\Pi_j u_{\tria_\vee}|_{H^1(\Omega_j)}^2
    \\
    &=\sum_{j=1}^m |u_{\tria_\vee}-\Pi_j
    u_{\tria_\vee}|_{H^1(\Omega)}^2 \eqsim \sum_{j=1}^m
    |u_{\tria_\vee}-u_{\tria_j}|_{H^1(\Omega)}^2.\qedhere
  \end{align*}
\end{proof}

Thanks to \eqref{eq:Ediff}, the latter result directly transfers to 
energy differences.  In fact, under the conditions of
 Theorem~\ref{thm:stable-splitting}, we have 
 \begin{align*}
   \mathcal{J}(\tria^\wedge)- \mathcal{J}(\tria_\vee) &\eqsim
   \sum_{j=1}^m \big( \mathcal{J}(\tria_j)- \mathcal{J}(\tria_\vee)
   \big).
 \end{align*}
This estimate is fundamental for our optimality analysis in
Section~\ref{sec:simpl-optim}. We make the following definition:
\begin{definition}
  \label{def:LDE}
  An energy~$\widetilde{\mathcal{J}}\,:\, \bbT \to \setR$ is said to
  satisfy the {\em lower diamond estimate} when for all lower diamonds
  $(\tria^\wedge,\tria_\vee;\tria_1, \dots, \tria_m)$ in $\bbT$, it holds that 
  \begin{align*}
    \widetilde{\mathcal{J}}(\tria^\wedge)-
    \widetilde{\mathcal{J}}(\tria_\vee) &\eqsim \sum_{j=1}^m \big(
    \widetilde{\mathcal{J}}(\tria_j)-
    \widetilde{\mathcal{J}}(\tria_\vee) \big),
  \end{align*}
  independent of the lower diamond.
\end{definition}
\begin{corollary}
  \label{Cor:triaLD}
  The energy $\mathcal{J}$ satisfies the lower diamond estimate.
\end{corollary}

\section{A posteriori error estimation}
\label{sec:error-bounds}

In this section we shall present a\new{n} edge-based variant of the standard
residual error estimator and recall some of its properties. To this end, we
fix some triangulation $\tria\in\bbT$ of $\Omega$.  For $S \in
\sides(\tria)$ we define $\Omega_\tria(S)$ as the interior of the
union of the triangles with common edge $S$, and we define the
\new{squared} local error indicators  by
\begin{align}\label{df:estimator} 
  \estT^2(S)&:= 
  \begin{cases}
    \displaystyle{ \sum_{{\{T \in \tria: T \subset
          \overline{\Omega_\tria(S)}\}}} h_T^2 \norm{ f}_{L^2(T)}^2 + h_S
      \bignorm{\jump{\nabla u_\tria}}_{L^2(S)}^2} &\quad \text{for $S
      \subset \Omega$}
    \\[2mm]
    \displaystyle{
          \sum_{{\{T \in \tria: T \subset \overline{\Omega_\tria(S)}\}}} h_T^2
      \norm{ f}_{L^2(T)}^2} &\quad \text{for $S \subset
      \partial \Omega$}.
  \end{cases}
\end{align}
Here $h_S := \abs{S}$,  $h_T := \abs{T}^{\frac 12}$, and
\begin{align*} 
  \jump{\nabla\uT}|_{S} &:=\sum_{\{T \in \tria: S \subset T\}}
  \nabla{\uT}|_{T} \cdot \bsn_{T}
\end{align*}
with $\bsn_{T}$ being the outward pointing unit normal on $\partial T$.
Note that, thanks to  the choice $h_T = \abs{T}^{\frac 12}$, we have
that the local mesh size $h_T$ decreases strictly with the factor
$2^{-1/2}$ at every bisection of $T$.

For $\tilde\sides\subset\sides(\tria)$ we define the accumulated
\new{squared} error \new{indicator} by 
\begin{align}
  \estT^2(\tilde\sides)&:=
  \sum_{S\in\tilde\sides}\estT^2(S).
\end{align}
For $\tria \in \bbT$ we define the \new{squared}
oscillation $\osc^2(\mathcal{U})$ on  \new{$\mathcal{U} \subset \tria$} by
\begin{align*}
  \osc^2(\mathcal{U}) &:= \sum_{T \in \mathcal{U}} h_T^2
  \norm{f-f_T}_{L^2(T)}^2,
\end{align*}
where $f_T := \frac{1}{\abs{T}} \int_T f\,dx$. 
It is well known that the estimator, defined in
\eqref{df:estimator}, is reliable and efficient in the following sense; compare e.g. with \cite{Verfuerth:96}.
\begin{proposition}\label{P:Bounds}
  For $\tria \in \bbT$ we have the bounds
  \begin{align*}
    \abs{u-\uT}_{H^1(\Omega)}^2& \Cleq\estT^2(\sides(\tria)) \Cleq
    \abs{u-\uT}_{H^1(\Omega)}^2+\osc^2(\tria).
  \end{align*}
\end{proposition}
This proposition shows that the  error estimator mimics the error
$\abs{u-\uT}_{H^1(\Omega)}$ up to oscillation.  

We define the \emph{total energy} $\mathcal{G}\,:\, \bbT \to \setR$ 
by
\begin{align}
  \label{eq:def_H}
  \mathcal{G}(\tria) :=\mathcal{J}(\tria) + \mathcal{H}(\tria),\quad\text{where}\quad
  \mathcal{H}(\tria) &:= \sum_{T \in \tria} h_T^2 \norm{f}_{L^2(T)}^2.
\end{align}
Note that $\mathcal{G}$, $\mathcal{H}$ and $\osc^2$ are non-increasing
with respect to $\bbT$. Since $\mathcal{H}(\tria)$, $\osc^2(\tria) \to
0$ for $\tria \to \tria^\top$, it is natural to set
$\mathcal{H}(\tria^\top) :=0$, $\osc^2(\tria^\top) := 0$, and 
$\mathcal{G}(\tria^\top) := \mathcal{J}(\tria^\top)$.

From Proposition~\ref{P:Bounds} and $\estT^2(\sides(\tria)) \geq
\mathcal{H}(\tria) \geq \osc^2(\tria)$, we obtain
\begin{align}
  \label{eq:energy_total_error}
  \estT^2(\sides(\tria)) &\eqsim \abs{u-\uT}_{H^1(\Omega)}^2
  +\osc^2(\tria)\eqsim \abs{u-\uT}_{H^1(\Omega)}^2
  +\mathcal{H}(\tria).
\end{align}
The term $(\abs{u-\uT}_{H^1(\Omega)}^2 +\osc^2(\tria))^{1/2}$ is referred to
as the {\em total error}. Similarly, we have in terms of an
energy difference, that 
\begin{align}
  \label{eq:energy+est+osc} \hspace{-8mm}
  \estT^2(\sides(\tria)) &\eqsim (\mathcal{J} +
  \osc^2)(\tria) - (\mathcal{J} +
  \osc^2)(\tria^\top) \eqsim \mathcal{G}(\tria) -
  \mathcal{G}(\tria^\top) \hspace{-10mm}
\end{align}
Therefore, in order to prove instance optimality for the total error, it
suffices to prove instance optimality of the energy difference of
the total  energy~$\mathcal{G}$.

In order two compare the energies of two discrete solutions, we need
the following lemma.
\begin{lemma}
  \label{lem:H_osc_discr}
  Let $\tria, \tria_* \in \bbT$ with $\tria \leq \tria_*$, then
  \begin{alignat*}{2}
    \mathcal{H}(\tria) - \mathcal{H}(\tria_*)& \eqsim
    \mathcal{H}(\tria \setminus \tria_*) &&:=\sum_{T \in
      \tria \setminus \tria_*} h_T^2 \norm{f}_{L^2(T)}^2.
    \\
    \osc^2(\tria) - \osc^2(\tria_*)& \eqsim \osc^2(\tria \setminus
    \tria_*) &&= \sum_{T \in \tria \setminus \tria_*} h_T^2
    \norm{f-f_T}_{L^2(T)}^2.
  \end{alignat*}
\end{lemma}
\begin{proof}
  Since every bisection locally reduces the mesh size by a factor of
  $2^{-1/2}$, we have \new{$\mathcal{H}(\tria_* \setminus \tria) \leq \frac{1}{2}
   \mathcal{H}(\tria \setminus \tria_*)$}. This and
  $\mathcal{H}(\tria) - \mathcal{H}(\tria_*) = \mathcal{H}(\tria
  \setminus \tria_*)- \mathcal{H}(\tria_* \setminus \tria)$ proves the
  claim for~$\mathcal{H}$. The proof of the second claim is similar using also 
  $\inf_{c \in \setR} \norm{f-c}_{L^2(T)} =
  \norm{f-f_T}_{L^2(T)}$.
\end{proof}

We shall now derive a discrete analogue of
Proposition~\ref{P:Bounds}. To this end, we need the 
Scott-Zhang type interpolation $\Pi_{\new{\tria_*\to \tria}}$ introduced in 
Section~\ref{sec:projection}.

\begin{lemma}\label{L:discreteBounds}
  Let $\tria,\tria_*\in\bbT$, with $\tria_*\ge\tria$ and denote by 
  $\sides=\sides(\tria)$ and $\sides_*=\sides(\tria_*)$ the respective
  sets of sides. 
  Then we have 
  \begin{align*}
    \abs{\uT-\uT[\tria_*]}^2_{H^1(\Omega)}&\Cleq\estT^2(\sides
    \setminus \sides_*) \Cleq
    \abs{\uT-\uT[\tria_*]}^2_{H^1(\Omega)}+\new{\mathcal{H}(\tria \setminus \tria_*)},
  \end{align*}
  where $\sides\setminus \sides_*$ is the set of sides in ${\sides}$
  that are refined in ${\sides_*}$.
\end{lemma}
\begin{proof}
  Let $e_*:=\uT[\tria_*]-\uT$, then by \new{Lemma~\ref{pro:proj}{\ref{itm:proj_V0}}}
we have that
  \begin{align*}
    \abs{e_*}_{H^1(\Omega)}^2&=\int_\Omega \nabla e_*\cdot\nabla e_*\,\dx  
    =\int_\Omega \nabla e_*\cdot(\nabla e_*- \nabla
    \Pi_{\tria_* \to \tria} e_*)\,\dx  
    \\
    &=\sum_{T\in \tria}\int_Tf(e_*-\Pi_{\tria_* \to \tria} e_*)\,\dx
    - \sum_{S\in\sides}\int_S
    \jump{\nabla \uT} (e_*-\Pi_{\tria_* \to \tria} e_*)\,\ds.
  \end{align*}
  It follows from Lemma~{\ref{pro:proj}\ref{itm:proj_compl}} that $e_* = \Pi_{\tria_* \to
    \tria} e_{\new{*}}$ on $\Omega \setminus \overline{\Omega(\tria \setminus
  \tria_*)}$.  
  \new{The first inequality to be shown follows by the trace theorem for the second sum, the Cauchy-Schwarz
    inequality and Lemma~{\ref{pro:proj}\ref{itm:proj_Omega_j}}.}

  In order to prove the \new{second inequality}, let $S\in {\sides\setminus
    \sides_*}$, i.e., $S$ is refined in $\sides_*$. In other words, we
  have for the midpoint $z$ of $S$, that {$z\in\nodes(\tria_*)$}. If
  $S\subset \partial\Omega$, then trivially have
  \begin{align}\label{eq:Sestimate}
    \estT^2(S)&\Cleq
    \abs{\uT-\uT[\tria_*]}^2_{H^1(\Omega_\tria(S))}+
    \sum_{{\{T \in \tria:T\subseteq \Omega_\tria(S)\}}}
     h_T^2 \norm{f}^2_{L^2(T)}.
  \end{align}
  For $S\not\subset\partial\Omega$, let $\tria':=\refine{\tria}{S}$, and
  let $\phi_z \in \VoT[\tria']$ be defined by $\phi_z(z)=1$, and
  $\phi_z(z')=0$ for $z' \in \nodes(\tria') \setminus \{z\}$. Note
  that $\phi_z \in \VoT[\tria_{*}]$, and $\supp \phi_z \subseteq
  \Omega_\tria(S)$.  We recall that $\jump{\nabla\uT}|_{S}\in\setR$ and
  deduce from \eqref{eq:discrete}, that
  \begin{align*}
    \frac12\int_S h_S \jump{\nabla \uT}^2\,\ds &= \int_S h_S
    \jump{\nabla \uT}^2\phi_z\,\ds
    \\
    &= -\int_{\Omega_\tria(S)} f h_S\jump{\nabla \uT}\phi_z\,\dx
    \\
    &\quad+ \int_{\Omega_\tria(S)}(\nabla \uT[\tria_*]-\nabla\uT) h_S
    \jump{\nabla \uT}\nabla \phi_z \,\dx
    \\
    &\le h_S \norm{f}_{L^2(\Omega_\tria(S))}\norm{\jump{\nabla
        \uT}\phi_z}_{L^2(\Omega_\tria(S))}
    \\
    &\quad+
    \norm{\nabla\uT[\tria_*]-\nabla\uT}_{L^2(\Omega_\tria(S))}\norm{
      \jump{\nabla \uT} h_S \nabla\phi_z}_{L^2(\Omega_\tria(S))}.
  \end{align*}
  With standard scaling arguments we obtain that
  \begin{align*}
    \bignorm{ \jump{\nabla
        \uT}{h_S}\nabla\phi_z}_{L^2(\Omega_\tria(S))}^2\Cleq \bignorm{
      \jump{\nabla \uT}\phi_z}_{L^2(\Omega_\tria(S))}^2\Cleq \int_S
    h_S \jump{\nabla \uT}^2\,\ds.
  \end{align*}
  Thus it follows from 
  Young's inequality that 
  \begin{align*}
    \int_S h_S \jump{\nabla \uT}^2\,\ds\Cleq 
    \norm{\nabla\uT[\tria_*]-\nabla\uT}_{L^2(\Omega_\tria(S))}^2
    +
    \sum_{{\{T \in \tria:T\subseteq \Omega_\tria(S)\}}}
     h_T^2 \norm{f}^2_{L^2(T)}.
  \end{align*}
  and consequently we have \eqref{eq:Sestimate} for all
  $S\in\sides\setminus \sides_*$.  Since we have at most a triple
  overlap of the $\Omega_\tria(S)$, $S\in\sides$, the assertion
  follows by summing over all $S\in \sides \setminus \sides_*$.
\end{proof}

The next result is the discrete analogue of~\eqref{eq:energy+est+osc}
and shows that the total energy~$\mathcal{G}$ matches
perfectly the  \new{squared} a posteriori error estimator defined in \eqref{df:estimator}.
\begin{proposition}\label{P:energy_gain}
  Let $\tria, \tria_*\in\bbT$ with $\tria \leq \tria_*$. Then we have
  \begin{align*}
    \mathcal{G}(\tria)-\mathcal{G}(\tria_*)&\eqsim
    \estT^2(\sides(\tria) \setminus \sides(\tria_*)).
  \end{align*}
\end{proposition}
\begin{proof}
\new{%
  By \eqref{eq:Ediff} and Lemma~\ref{lem:H_osc_discr}, we have
$\mathcal{G}(\tria)-\mathcal{G}(\tria_*) \eqsim
\abs{\uT-\uT[\tria_*]}^2_{H^1(\Omega)}+\mathcal{H}(\tria \setminus \tria_*)$.
From
  \begin{align}
    \label{eq:Omega_vs_Omega_T}
    \overline{\Omega(\tria \setminus \tria_*)} &=
    \bigcup_{ T \in\tria \setminus \tria_*} T =
    \bigcup_{ S \in\sides(\tria) \setminus \sides(\tria_*)}
    \overline{\Omega_\tria(S)},
  \end{align}
  with an at most triple overlap of the sets $\overline{\Omega_\tria(S)}$,
  it follows that 
  $\mathcal{H}(\tria \setminus \tria_*) \lesssim \estT^2(\sides(\tria)
    \setminus \sides(\tria_*))$. An application of Lemma~\ref{L:discreteBounds}  completes the proof.}
\end{proof}

Let us turn to the case of coarsenings in mutual disjoint areas.  As a direct
consequence of Lemma~\ref{lem:H_osc_discr} and the fact that
$\mathcal{J}$ satisfies the lower diamond estimate
(Corollary~\ref{Cor:triaLD}), we get the
following result.
\begin{corollary}
  \label{cor:G_LDE}
  The energies $\mathcal{H}$, $\osc^2$ and $\mathcal{G}$ satisfy the
  lower diamond estimate.
\end{corollary}

\begin{remark}
  In this paper we resort to edge based error indicators
  \eqref{df:estimator} for the following reason. 
  In the situation of
  Proposition~\ref{P:energy_gain} consider e.g. the 
  element based \new{squared} error indicators
  \begin{align*}
    \new{\tilde{\mathcal E}_\tria^2(T)}:= h_T^2 \norm{f}^2_{L^2(T)}+ \tfrac 12
    h_T\norm{\jump{\nabla \uT}}^2_{L^2(\partial T)}
  \end{align*}
  from \cite{Verfuerth:96}. Then we have the estimate 
  \begin{align*}
    \new{\tilde{\mathcal E}_\tria^2(T)}\Cleq \sum_{S\subset
      T}\norm{\nabla\uT-\nabla\uT[\tria_*]}^2_{L^2(\Omega_\tria(S))} +
    h_T^2 \norm{f}^2_{L^2(\Omega_\tria(S))}
  \end{align*}
  only if all three edges of $T$ are bisected at least once in
  $\tria_*$; compare e.g. with \cite{Doerfler:96,MoNoSi:00,MoNoSi:02}. 
  Since $\tria_*$ is conforming, this can only be true for all elements $T\in{\tria\setminus
    \tria_*}$ \new{when} all elements of $\tria$ are at least refined
  twice in $\tria_*$.
  \new{Consequently}, either we have estimates
  similar to those in Proposition~\ref{P:energy_gain} for \new{squared} element based error
  \new{indicators} running over different sets of elements for both inequalities
  respectively, or we need to resort to global refinement.
  In the latter case we have ${\tria\setminus \tria_*}=\tria$.
\end{remark}

Our optimality proof later will be based on the language of
populations. Naturally, we define $\mathcal{G}(\Pop) :=
\mathcal{G}(\tria(\Pop))$ for $\Pop \in \bbPhat$. Now, let us reformulate
our error estimator estimates in terms of populations. 

Due to the one-to-one correspondence of $\sides(\tria(\Pop))$ and
$\Pop\plus \setminus \Pop$, see~\eqref{eq:midpoints_pre}, we set for
${\mathcal U} \subset \Pop\plus\setminus \Pop$
\begin{align*}
  \mathcal{E}^2_\Pop(\mathcal{U}):=\mathcal{E}^2_{\tria(\Pop)}
  (\textsf{midpts}^{-1}({\mathcal U})).
\end{align*}
This allows us to rewrite Proposition~\ref{P:energy_gain}
as follows.
\begin{corollary}
  \label{C:energy_gain_P}
  Let $\Pop, \Pop_*\in\bbP$ with $\Pop \leq \Pop_*$. Then we have
  \begin{align*}
    \mathcal{G}(\Pop)-\mathcal{G}(\Pop_*)&\eqsim
    \estP^2\big(\Pop_* \cap (\Pop\plus \setminus \Pop)\big).
  \end{align*}
\end{corollary}

\begin{remark}[Upper diamond estimate] 
  Let $(\tria^\wedge,\tria_\vee;\tria_1, \dots, \tria_m)$ be an upper diamond in $\bbT$.
  Since the requirement of the areas of
  refinement $\Omega(\tria^\wedge \setminus \tria_j)$ being mutually
  disjoint is equivalent to the requirement that the sets
  $\tria^\wedge \setminus \tria_j$, or the sets
  $\sides(\tria^\wedge) \setminus \sides(\tria_j)$ being mutually
  disjoint, from Proposition~\ref{P:energy_gain} we obtain
  that
 \begin{align*}
   \mathcal{G}(\tria^\wedge) - \mathcal{G}(\tria_\vee) &\eqsim 
   \estT[\tria^\wedge]^2(\sides(\tria^\wedge)\setminus\sides(\tria_\vee))
   =\sum_{j=1}^m
   \estT[\tria^\wedge]^2(\sides(\tria^\wedge)\setminus\sides(\tria_j))
   \\
   &\eqsim
   \sum_{j=1}^m
   \big( \mathcal{G}(\tria^\wedge) - \mathcal{G}(\tria_j)\big).
 \end{align*}
\end{remark}

\section{The adaptive finite element method (AFEM)}
\label{sec:afem}

According to \cite{BabuskaRheinboldt:78}, the maximum marking
strategy, marks sides for refinement that correspond to \new{squared} local error indicators that
are not less than some constant multiple $\mu \in (0,1]$ of the
maximum \new{squared} local error indicator. 

In view of the fact that, in order to refine a side, generally more
sides have to be bisected to retain conformity of the triangulation, we
consider the following {\em modified maximum marking strategy}: First
we determine a side $S$ such that the sum $\bar{\mathcal{E}}^2$ of
\emph{all} local
\new{squared} error indicators of the
sides that have to be bisected in order to refine $S$ is
maximal. Then, in some arbitrary order, running  over the sides in the
triangulation, we mark those sides $\tilde{S}$ for refinement for which
the sum  of all
\new{squared} local error indicators that correspond to the sides that have to be bisected
in order to refine $\tilde{S}$, but that   do not have to be bisected for the
refinement of sides that are marked earlier, is not less than
$\mu\bar{\mathcal{E}}^2$. 

To give a formal description, 
for $\tria \in \bbT$ and $S \in \sides(\tria)$, let
$$
\refd{\tria}{S}:=\sides(\tria) \setminus \sides(\refine{\tria}{S}),
$$
being the subset of sides in $\sides(\tria)$ that are bisected when
passing to  the smallest refinement (in $\bbT$) of $\tria$ in which
$S$ has been bisected. 
Then the adaptive finite element method reads as follows:

\begin{algorithm}[AFEM]\label{algo:AFEM}
  Fix $\mu \in (0,1]$ and set 
  $\tria_0:=\tria_\bot$ and $k=0$. The adaptive loop is an iteration
  of the following steps:
 \par{  
  \linespread{1.5}
    \sffamily
   \renewcommand{\arraystretch}{2}
    \begin{tabbing}
      (1) {ESTIMATE:}\quad\= \kill
       (1)  {SOLVE}: \>compute $\uT[\tria_k]\in\VoT[\tria_k]$;\\ %
       (2) {ESTIMATE}: \> compute $\{\estT[\tria_k]^2(S) : S \in \sides(\tria_k)\}$;
       \\
       (3) {MARK}: \> $\estTbar[\tria_k]^2
        :=\max\set{\estT[\tria_k]^2(\refd{\tria_k}{S}): S \in
          \sides(\tria_k)}$;
      \\
      \>$\Mk:=\emptyset$; ${\mathcal
        C}_k:=\sides(\tria_k)$; $\tildeMk:=\emptyset$;
      \\
      \>while \= ${\mathcal C}_k \neq \emptyset$ do 
      \\
      \>\>select $S \in {\mathcal C}_k$;
      \\
      \> \>if $\estT[\tria_k]^2(\refd{\tria_k}{S} \setminus \tildeMk)
      \geq \mu \estTbar[\tria_k]^2$;
      \\
      \>\> then \=$\Mk:=\Mk \cup \{S\}$;\\
      \>\>\>$\tildeMk:=\tildeMk \cup \refd{\tria_k}{S}$;\\
      \>\> end if;\\
      \>\>${\mathcal C}_k:={\mathcal C}_k \setminus \refd{\tria_k}{S}$;\\
      \>end while;\\
      (4) {REFINE}: \>compute $\tria_{k+1}=\refine{\tria_k}{\Mk}$ and 
      increment $k$.
    \end{tabbing}}
  \smallskip
  
  If we define $\Pop_k:= \Pop(\tria_k)$, then we can rewrite our
  algorithm also in  the language of populations:  
  \par{  
    \linespread{1.5}
    \sffamily
    \begin{tabbing}
      (1) {ESTIMATE:}\quad\= \kill
      (1)  {SOLVE}: \>compute $\uT[\tria_k]\in\VoT[\tria_k]$;\\ %
      (2) {ESTIMATE}:  \>compute  $\{\estP[\Pop_k]^2(P) : P
      \in \Pop_k\plus \setminus \Pop_k\}$;
      \\
      (3) {MARK}:      \>$\estPbar[\Pop_k]^2 :=\max\set{\estP[\Pop_k]^2((\Pop_k \oplus P)
        \setminus \Pop_k)\,:\, P 
        \in \Pop_k\plus \setminus \Pop_k)}$;
      \\
\>$\Mk:=\emptyset$; ${\mathcal C}_k:= \Pop_k\plus
      \setminus \Pop_k$; $\tildeMk:=\emptyset$;
      \\
      \>while \= ${\mathcal C}_k \neq \emptyset$ do 
      \\
      \>\>select $P \in {\mathcal C}_k$;\\
      \> \>if $\estP[\Pop_k]^2((\Pop_k \oplus P) \setminus (\Pop_k \cup
      \widetilde{\mathcal{M}}_k)) \geq \mu \estPbar[\Pop_k]^2$;
      \\
      \>\> then \=$\Mk:=\Mk \cup \{P\}$;\\
      \>\>\>$\tildeMk:=\tildeMk \cup ((\Pop_{\new{k}} \oplus P) \setminus \Pop_{\new{k}})$;\\
      \>\> end if;\\
      \>\>${\mathcal C}_k:={\mathcal C}_k \setminus \big((\Pop_{\new{k}} \oplus P)
      \setminus \Pop_{\new{k}}\big)$;\\
      \>end while;\\
      (4) {REFINE}: \>$\Pop_{k+1}:=\Pop_k \oplus \Mk\; \big[\!\!= \Pop_k \cup \tildeMk\big]$ and
      increment $k$.
    \end{tabbing}}
\end{algorithm}

\begin{proposition}
  \label{prop:lwb}
  For the sequences $(\Pop_k)_{k \in \mathbb{N}_0}$ and
  $(\mathcal{M}_k)_{k \in \mathbb{N}_0}$ produced by
  Algorithm~\ref{algo:AFEM} (second formulation),
  we have $\mathcal{M}_k \neq \emptyset$ and
  \begin{align*} 
    \estP[\Pop_k]^2 \big( \Pop_{k+1} \cap (\Pop_k\plus \setminus
    \Pop_k) \big) = \estP[\Pop_k]^2 \big( (\Pop_k \oplus \Mk)
    \setminus \Pop_k \big) \geq \mu \,\# \Mk\,
    \estPbar[\Pop_k]^2.
  \end{align*}
\end{proposition}
\begin{proof}
  Consider the while-loop in MARK. As long as $\Mk=\emptyset$, we have
  $\tildeMk=\emptyset$. \new{Thus for every $P \in \Pop_k\plus
      \setminus \Pop_k$ that has been considered and all $P' \in (\Pop_k \oplus
      P) \setminus \Pop_{k}$, we conclude
      $\estP[\Pop_k]^2((\Pop_k \oplus P')
      \setminus \Pop_k)\le\estP[\Pop_k]^2((\Pop_k \oplus P) \setminus
      \Pop_k) < \mu 
      \estPbar[\Pop_k]^2$.} 
  Hence assuming that $\Mk$ remains empty, at  
  some moment $P
  \in \Pop_k\plus \setminus \Pop_k$ is encountered with
  $\estP[\Pop_k]^2((\Pop_k \oplus P) \setminus (\Pop_k \cup \tildeMk)) =
  \estPbar[\Pop_k]^2$, which yields a contradiction. Therefore, after
  termination of the while-loop in MARK, we have $\Mk \neq
  \emptyset$.

  Each time a $P$ is added to $\Mk$, the quantity
  $\estP[\Pop_k]^2((\Pop_k \oplus \Mk) \setminus \Pop_k)$
  increased by at least $\mu \estPbar[\Pop_k]^2$, which
  shows $\estP[\Pop_k]^2 \big( (\Pop_k \oplus \Mk)
    \setminus \Pop_k \big) \geq \mu \, \# \Mk\,
    \estPbar[\Pop_k]^2.$

  Since $\Mk \subset \Pop_k\plus$ and $\Pop_{k+1} = \Pop_k \oplus
  \Mk$, we have $\Pop_{k+1} \wedge \Pop_k\plus= \Pop_k \oplus
  \Mk$. Thus 
  \begin{align*}
    \Pop_{k+1} \cap (\Pop_k\plus \setminus \Pop_k) = (\Pop_{k+1}
    \wedge \Pop_k\plus) \setminus \Pop_k = (\Pop_k \oplus \Mk) \setminus
    \Pop_k = \tildeMk,
  \end{align*}
  which concludes the proof.
\end{proof}

\begin{remark} \new{Generally, the set of marked edges or persons determined in MARK depend on the ordering in which the sets $\sides(\tria_k)$ or $\Pop_k\plus
      \setminus \Pop_k$ are traversed.} \new{In particular,
      the marking strategy does not necessarily mark the
      $S\in\sides(\tria)$ with
      $\estT[\tria_k]^2(\refd{\tria_k}{S})=\estTbar[\tria_k]$ for refinement.}
\end{remark}

\begin{remark} \new{It is not difficult to see that ESTIMATE
    can be implemented in ${\mathcal O}(\#\Pop_k)$ operations. Also
    REFINE can be implemented in ${\mathcal O}(\#\Pop_k)$ operations,
    once $\Mk$ is determined by MARK; recall $\Pop_{k+1} = \Pop_k
    \oplus \Mk = \Pop_k \cup \tildeMk$. In Appendix~\ref{appendix} it
    is demonstrated that the same holds true for a slightly modified
    MARK, which yields a set $\Mk$,
     which 
    has qualitatively the same properties as given in
    Proposition~\ref{prop:lwb}. Consequently, our instance optimality
    result still holds with this modified marking.} 
\end{remark}

\section{Fine properties of populations}
\label{sec:fine}

Before we get to our optimality proof, we need some fine properties
\new{of} populations.

\subsection{Ancestors, descendants and free elements}

For $P \in \Pop^\top$ we define its {\em set of ancestors
  $\ancestor(P)$} as follows: If $\generation(P)=0$, then
$\ancestor(P) := \emptyset$.  For $\generation(P) \geq 1$, we define
inductively
\begin{align*}
  \ancestor(P) &:= \parents(P) \cup \bigcup_{Q \in \parents(P)}
  \ancestor(Q).
\end{align*}
Moreover, we denote the set of the {\em descendants} of~$P$
 by
\begin{align*}
  \descendant(P) := \set{ P' \in \Pop^\top
    \,:\, P \in \ancestor(P')}.
\end{align*}

As a shorthand notation, we write $P'\parentof P$ or $P\childof P'$,
when $P \in \children(P')$, or equivalently, $P' \in \parents(P)$; and
$P' \ancestorof P$ or $P\descendantof P'$, when $P \in
\descendant(P')$, or equivalently, $P' \in \ancestor(P)$.

For $\mathcal{U} \subset \Pop^\top$ we define
\begin{align*}
  \children(\mathcal{U}) &:= \bigcup_{P \in \mathcal{U}}
  \children(P), &
  \parents(\mathcal{U}) &:= \bigcup_{P \in \mathcal{U}}
  \parents(P),\\
   \ancestor(\mathcal{U}) &:= \bigcup_{P \in \mathcal{U}}
  \ancestor(P), &
  \descendant(\mathcal{U}) &:= \bigcup_{P \in \mathcal{U}}
  \descendant(P).
\end{align*}

If $P, P'  \in \Pop^\top$, with $P \neq P'$, have a joint child, then we
call them {\em partners}, and write 
 $P\partnerof P'$. 

For $k\in\setN_0$, we define
$$
\generation^{-1}(k):=\{P\in\Pop^\top\colon\generation(P)=k\}.
$$ 

The following lemma summarises some apparent basic properties 
without proof.
\begin{lemma}
  \label{lem:anc} For $\Pop, \Pop_*\in \bbP$, we have
  \begin{enumerate}
  \item $\ancestor(\Pop) \subset\Pop$;
  \item if $\Pop\leq \Pop_*$, then $\ancestor(\Pop) \cap (\Pop_* \setminus
    \Pop) = \emptyset$;
  \item if $\Pop\leq \Pop_*$, then $\descendant(\Pop_* \setminus \Pop) \cap
    \Pop = \emptyset$;
  \item \label{itm:anc4} if $\mathcal{C} \subset \Pop^\top$, then $\Pop
    \oplus \mathcal{C} =\Pop \cup \ancestor(\mathcal{C}) \cup
    \mathcal{C}$.
  \end{enumerate}
\end{lemma}

So far we have stated only very general properties of populations.
However, populations correspond
to conforming triangulations created by newest vertex bisection of the initial
triangulation $\tria_\bot$. 
In the following we shall exploit the structures inherited by 
this fact in order to prove much stronger results.

The following lemma shows that the number of ancestors of the same
generation is for every person bounded by a uniform constant.
For apparent reasons we call this property {\em limited genetic diversity (LGD)}.

\begin{proposition}
  \label{pro:ILS}
  We have
  \begin{align*}
    \sup_{P \in \Pop^\top} \sup_{k \in \setN}\# \big( \ancestor(P)
    \cap \generation^{-1}(k) \big) &=: \cGD < \infty.
  \end{align*}
  The scalar~$\cGD$ is called the {\em genetic diversity constant}.
\end{proposition}
\begin{proof}
Thanks to the refinement by bisection, 
for $T \in \mathfrak{T}$, we have $|T| \eqsim 2^{-\generation(T)}$.
Consequently, by the uniform shape regularity of $\mathfrak{T}$, 
for $P' \in \Pop^\top$
and $P \in \children(P')$, we have that ${\rm dist}(P',P) \eqsim
2^{-\generation(P')/2}$. 
Applying a geometrical series argument, we thus infer that for $P' \in
\ancestor(P) \cap \generation^{-1}(k)$, ${\rm dist}(P',P) \eqsim
2^{-k/2}$. 

Again by the uniform shape regularity of $\mathfrak{T}$, any ball of radius $2^{-k/2}$
contains at most an uniformly bounded number of vertices  
of the uniform refinement of $\tria_\bot$ with triangles of generation
$k$. This completes the proof.  
\end{proof}

The following property of the newest vertex bisection is even more
peculiar:
\begin{proposition}
  \label{pro:smallfamily}
  Let $P_1, P_2 \in \Pop^\top$  be partners with $\generation(P_1)
  \geq 2$. Then $P_1$ and $P_2$ have a joint parent.
\end{proposition}
\begin{proof}
  Let $k:={\rm gen}(P_1)$ $(={\rm gen}(P_2))$, and let $P$ be a child
  of $P_1$ and $P_2$.  A patch of the coarsest triangulation in $\bbT$
  that contains $P$ looks as indicated in Figure~\ref{fig1}.
  \begin{figure}[h]
    \centering
    \input{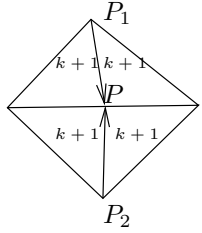}
    \caption{Coarsest triangulation that contains $P$}
    \label{fig1}
  \end{figure}
  Here, and in the following figures, the arrows indicate the
  parent-child relationships, and the numbers indicate the generations
  of the triangles.

  The two possible patches (up to symmetries) of coarsest triangulations
  in $\bbT$ that contain $P_1$ and $P_2$ look as indicated in
  Figure~\ref{fig2}.
  \begin{figure}[h]
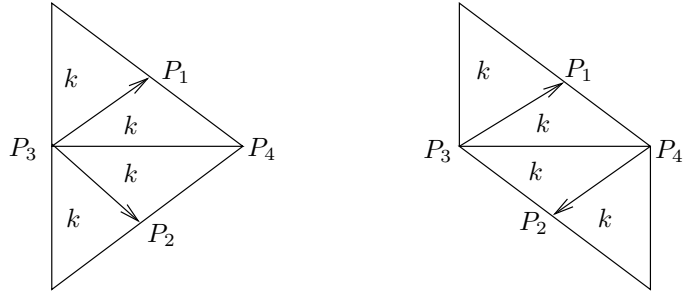

    \centering
    \input{pic3.pspdftex}
    \hspace{2cm}
    \input{pic2.pspdftex}
    \caption{Two possible coarsest triangulations that contain $P_1$ and $P_2$}\label{fig2}
  \end{figure}
  In the left picture, $P_1$ and $P_2$ have a joint parent $P_3$. 

  In the right picture, $P_1$ has parent $P_3$, and $P_2$ has parent
  $P_4$. \new{Since $P_3$ and $P_4$ are vertices of a joint $T \in \mathfrak{T}$ and have the same generation, their generation must be zero, i.e., ${\rm gen}(P_1)=1$.}
\end{proof}

The next lemma shows that any two ancestors of some person that have the same (non-zero) generation  are linked via a sequence of partners.
\begin{lemma}
  \label{lem:chainpartnermodified}
  Let $P \in \Pop^\top$ and $\bar P, \tilde P \in \ancestor(P)$,
  $\tilde P\neq \bar P$, with
  $\generation(\tilde P) = \generation(\bar P) \geq 1$. Then, for some $1 \leq m \leq  \generation(P) -
    \generation(\bar P)$, there exist $P_0,\dots, P_m \in \ancestor(P)$ such that $P_0 \partnerof \cdots
  \partnerof P_m$ and $P_0=\bar P$ and $P_m=\tilde P$. 
  \new{In particular, $\children(P_{i-1})\cap\children(P_i)\cap(\ancestor(P) \cup \{P\}) \neq\emptyset$,
    $i=1,\ldots,m$.}
\end{lemma}
\begin{proof}
  Fix $P \in \Pop^\top$.  We prove the claim by induction over
  $k:=\generation(P)-\generation(\bar P)$.
   If $k=1$, then $\bar P \neq \tilde P$ have the joint child $P$, and hence 
   \new{$\children(\bar P)\cap\children(\tilde P )\cap \{P\} \neq\emptyset$.}
   
  Now let $k \geq 2$,  and assume that the claim is already true
  for $k-1$. Let $\bar P' \in \children(\bar P) \cap \ancestor(P)$ and
  $\tilde P' \in
  \children(\tilde P) \cap \ancestor(P)$.
  If $\bar P'=\tilde P'$, then \new{$\children(\bar P')\cap\children(\tilde P')\cap \ancestor(P)) \neq\emptyset$}.

  Otherwise, by induction for some $1 \leq m-1 \leq k-1$, there
  exist
   $\bar P_0, \ldots,\bar
  P_{m-1}\in \ancestor(P)$ such that
  $\bar P_0 \partnerof \cdots \partnerof \bar P_{m-1}$
  and  $\bar P_0=\bar P'$ and $\bar P_{m-1}=\tilde P'$. 
  
  Because of $\generation(\bar P_i) = \generation(\bar P)+1 \geq 2$,
  by Proposition~\ref{pro:smallfamily} there exist $P_1, \dots,
  P_{m-1}$ such that for $i=1, \dots,m-1$, $P_i$ is a parent of $\bar
  P_{i-1}$ and $\bar P_i$; in particular $P_i \in
  \ancestor(P)$.

  Setting, $P_0:=\bar P$ and $P_m := \tilde P$, we have found a
  sequence in $\ancestor(P)$ such that subsequent persons have a joint
  child \new{in $\ancestor(P)$}. 
  By removing \new{possibly} subsequent equal persons, we have found a sequence with the
  required properties.
\end{proof}

\begin{definition}
  We say that a set $\mathcal{U} \subset \Pop^\top$ is {\em
    descendant-free} when $\descendant(\mathcal{U}) \cap \mathcal{U} =
  \emptyset$.
\end{definition}

The next proposition generalizes upon Proposition~\ref{pro:ILS}.
\begin{proposition}
  \label{pro:descendantfree}
  Let $P \in \Pop^\top$ and\/ $\mathcal{U} \subset \ancestor(P)
  \setminus \Pop_\bot$. If $\mathcal{U}$ is descendant-free, then $\#
  \mathcal{U} \leq \cGD$.
\end{proposition}
\begin{proof}
  Since $\#\ancestor(P)< \infty$, there are
  at most finitely many descendant-free subsets  of $\ancestor(P) \setminus
  \Pop_\bot$. 
  Among them, let~$\mathcal{U}$ denote the one that first
  maximizes~$\#\mathcal{U}$ and then  $\sum_{Q \in \mathcal{U}}
  \generation(Q)$. We shall show that this implies that all persons
  in~$\mathcal{U}$ are of the same 
  generation. Therefore, by Proposition~\ref{pro:ILS}, we conclude the
  claim~$\#\mathcal{U} \leq \cGD$.  
  
  Let $\generation(P)\geq 2$, so that $\mathcal{U} \neq \emptyset$.
  Define $k := \new{\min} \set{\generation(P')\,:\,P' \in
    \mathcal{U}}$.  \new{In order to}
  show that $\mathcal{U} \subset \generation^{-1}(k)$,  we
  proceed by contradiction and assume that $\mathcal{U} \not\subset
  \generation^{-1}(k)$, \new{i.e., there exists a $Q \in \mathcal{U}$ with
  $\generation(Q)>k$,
  \new{and so $\generation(P)>k+1$}. 
  Since $\mathcal{U}$ is descendent-free, there exists a} $\tilde P \in
  (\new{\ancestor}(Q) \cap \ancestor(P) \cap \generation^{-1}(k))
  \setminus \mathcal{U}$.

  By definition of~$k$ there exists $\bar P\in \mathcal{U}$
  with~$\generation(\bar P)=\generation(\tilde P) =k$.  Due to
  Lemma~\ref{lem:chainpartnermodified} we find a finite sequence of
  partners in $\ancestor(P)$, starting with $\bar P$ and ending
  with~$\tilde P$, \new{where each couple has a common child in $\ancestor(P)$}.  Since $\tilde P \not\in \mathcal{U}$, we can
  select from this sequence a couple $\bar P' \partnerof
  \tilde P'$, with $\bar P' \in \mathcal{U}$, $\tilde P' \not\in
  \mathcal{U}$, \new{$\generation(\tilde P')=\generation(\bar P')=k$}, \new{and that has a common child $\hat P' \in \ancestor(P)$}.

   \new{On the one hand, since $\mathcal{U}$ is descendant-free, $\bar
  P'\in\mathcal{U}$ and $\hat P'\childof \bar P'$, we conclude that 
  $\descendant(\hat P')\cap\mathcal{U}\subset \descendant(\bar P')\cap\mathcal{U}=\emptyset$.
  On the other hand, thanks to the definition of $k$ and $\tilde
  P'\not\in\mathcal{U}$,  we have that $\bar P'$ is the only ancestor
  of $\hat P'$ in $\mathcal{U}$. In other words,
  $\mathcal{U}':=(\mathcal{U}\setminus\{\bar P'\})\cup\{\hat P'\}$ is
  a descendent-free subset of $\ancestor(P)$. Since $\#
  \mathcal{U}'= \# \mathcal{U}$ and $\sum_{Q \in \mathcal{U}'}
  \generation(Q)=1+ \sum_{Q \in \mathcal{U}} \generation(Q)$, this is
  the desired contradiction.}
  \end{proof}

\begin{definition}
  \label{def:free}
  Let $\mathcal{U} \subset \Pop^\top\setminus \Pop_\bot$. We call the
  subset 
  \begin{align*}
    \free(\mathcal{U}) := \set{ P \in \mathcal{U} \colon
      \descendant(P) \cap \mathcal{U} = \emptyset}. 
  \end{align*}
  the set of {\em free persons} in $\mathcal{U}$.
\end{definition}
The following lemma collects some basic properties of free subsets.
\begin{lemma}
  \label{lem:free}
  Let $\mathcal{U} \subset \Pop^\top \setminus \Pop_\bot$. 
  \begin{enumerate}
  \item \label{itm:free1} The set $\free(\mathcal{U})$ is descendant-free.
  \item \label{itm:free1b} If $\mathcal{U}$ is descendant-free, then
    $\free(\mathcal{U}) = \mathcal{U}$.
  \item \label{itm:free2} If $\#\mathcal{U}<\infty$, then
    $\mathcal{U} \cup\ancestor(\mathcal{U}) =
    \free(\mathcal{U}) \cup \ancestor(\free(\mathcal{U}))$. 
     \item \label{itm:free4} If $\#\mathcal{U}<\infty$ and
         $\Pop \in \bbP$, then $\Pop \oplus \mathcal{U}=\Pop \oplus
         \free(\mathcal{U})$.
  \item \label{itm:free3} If $\#\mathcal{U}<\infty$ and $\mathcal{U}
    \neq \emptyset$, then $\free(\mathcal{U})\neq \emptyset$.
  \end{enumerate}
\end{lemma}
\begin{proof}
  \ref{itm:free1}: Let $P \in \free(\mathcal{U})$,
  then by definition we have $\descendant(P) \cap \mathcal{U} = \emptyset$.
  Since $\free(\mathcal{U}) \subset \mathcal{U}$, we conclude that 
  $\descendant(P) \cap \free(\mathcal{U}) = \emptyset$,
  i.e., $\free(\mathcal{U})$ is descendant-free.

  \ref{itm:free1b}: The claim follows directly from the assumption
  $\descendant(\mathcal{U}) \cap \mathcal{U} =\emptyset$.
  
   \ref{itm:free2}: Obviously, it is sufficient to prove $\mathcal{U}
  \cup\ancestor(\mathcal{U})\subset \free(\mathcal{U}) \cup
  \ancestor(\free(\mathcal{U}))$. 
  Let $P \in \mathcal{U}$. If $P \in \free(\mathcal{U})$, then $P \cup
  \ancestor(P) \subset
  \free(\mathcal{U})\cup\ancestor(\free(\mathcal{U}))$. 
  Otherwise, if $P \not\in \free(\mathcal{U})$, then pick a $P' \in \mathcal{U}
  \cap \descendant(P)$. If $P' \not\in \free(\mathcal{U})$, then,
  because $\#\mathcal{U}<\infty$,  by continuing this process, after
  finitely many steps a descendant $P''$ of $P'$, and thus of $P$, is
  found, which is in $\free(\mathcal{U})$. 
  We conclude that $P \cup \ancestor(P) \subset \ancestor(P'') \subset
  \ancestor(\free(\mathcal{U}))$, which finishes the proof.

  \ref{itm:free4}: By \ref{itm:free2} and
  Lemma~\ref{lem:anc}\,\ref{itm:anc4}, we have $\Pop \oplus \mathcal{U}=\Pop
  \cup \mathcal{U} \cup \ancestor(\mathcal{U}) =
  \new{ \Pop \cup \free(\mathcal{U}) \cup \ancestor(\free(\mathcal{U}))=}
   \Pop \oplus
  \free(\mathcal{U})$.

 \ref{itm:free3}: Let $\# \mathcal{U} < \infty$
  with $\mathcal{U} \neq \emptyset$. Then
  $\free(\mathcal{U})=\emptyset$ together with~\ref{itm:free2} implies
  $\mathcal{U} = \emptyset$,
  which contradicts~$\mathcal{U} \neq \emptyset$.
\end{proof}

The following lemma states that removing free persons from a population results in a (smaller) population.

\begin{lemma} 
  \label{lem:free_removed}
  Let $\Pop_*, \Pop \in \bbP$ 
    Then,
 \begin{enumerate}
  \item \label{item1} for ${\mathcal C} \subset \free(\Pop_* \setminus \Pop_\bot)$, we have $\Pop_* \ominus {\mathcal C} =\Pop_*\setminus {\mathcal C}$;
  \item \label{item2} $\free(\Pop_* \setminus \Pop) \subset \free(\Pop_*\setminus \Pop_\bot)$.
  \end{enumerate}
  \end{lemma}
  
\begin{proof} \mbox{}
 \ref{item1} By assumption, $\Pop_*\setminus {\mathcal C}
  \in \bbP$ and thus $\Pop_* \ominus {\mathcal C} =\Pop_*\setminus {\mathcal C}$.

\ref{item2} 
 \new{By $\free(\Pop_* \setminus \Pop) \subset \Pop_* \setminus \Pop$, 
 we have $\free(\Pop_* \setminus \Pop)  \cap \Pop =\emptyset$.
  Since $\Pop \in \bbP$, the latter shows that $\descendant(\free(\Pop_* \setminus \Pop) )\cap \Pop =\emptyset$.
  Finally, from from  $\descendant(\free(\Pop_*\setminus \Pop))
  \cap (\Pop_* \setminus \Pop) = \emptyset$,
   we conclude that $\descendant(\free(\Pop_* \setminus \Pop) ) \cap \Pop_*=\emptyset$.}
 \end{proof}

For $\Pop \in \bbP$, the set $\free(\Pop \setminus \Pop_\bot)$
  are the nodes of $\tria(\Pop)$ that are ``free'' in the sense that
  they can be removed while retaining a conforming triangulation,
  i.e., a triangulation in $\bbT$. 
  Remarkably, as follows from the following proposition, 
  the number of free nodes in any triangulation in $\bbT$ cannot be reduced by more than some constant factor
  by whatever further refinement in $\bbT$.
  This proposition plays a crucial role
  in the optimality proof in Section~\ref{sec:simpl-optim}. 
  
  \begin{theorem}
  \label{thm:free_bound}
  Let $\mathcal{U} \subset \mathcal{V} \subset \Pop^\top \setminus
  \Pop_\bot$ with $\# \mathcal{V}<\infty$. Then
  \begin{align*}
    \# \free(\mathcal{U}) &\leq \cGD\, \#\free(\mathcal{V}) .
  \end{align*}
\end{theorem}
\begin{proof}
  It follows from $\free(\mathcal{U}) \subset \mathcal{U}\subset
  \mathcal{V}$ and Lemma~\ref{lem:free}\,\ref{itm:free2}, applied to
  $\mathcal{V}$, that
  \begin{align*}
    \free(\mathcal{U}) \subset \free(\mathcal{V}) \cup
    \ancestor(\free(\mathcal{V})).
  \end{align*}
  Thus we can write $\free(\mathcal{U})$ as
  \begin{align*}
    \free(\mathcal{U}) &= \bigcup_{P \in \free(\mathcal{V})}
    \underbrace{\Big( \big( \set{P} \cup \ancestor(P)\big) \cap
      \free(\mathcal{U}) \Big)}_{=:\mathcal{V}_P}.
  \end{align*}
   Now the claim follows, when  $\#\mathcal{V}_P \leq
  \cGD$ for all $P \in \free(\mathcal{V})$.  To this end, let $P \in
  \free(\mathcal{V})$. 
  
  Assume first that $P \in \free(\mathcal{U})$.  Since
  $\free(\mathcal{U})$ is descendent-free
  we have $\ancestor(P) \cap
  \free(\mathcal{U}) = \emptyset$. Thus $\mathcal{V}_P = \set{P}$
  and $\# \mathcal{V}_P =1 \leq \cGD$.

  Now assume $P \not\in \free(\mathcal{U})$. Then
  $\mathcal{V}_P = \ancestor(P) \cap \free(\mathcal{U})$. Since
  $\free(\mathcal{U})$ is descendant-free, the subset $\mathcal{V}_P$
  of $\ancestor(P)$ is descendant-free, and it follows by
  Proposition~\ref{pro:descendantfree} that $\#\mathcal{V}_P \leq
  \cGD$.
\end{proof}

\subsection{Populations and the lower diamond estimate}
\label{ssec:pop_and_LDE}

In this subsection we shall translate the lower diamond estimate of
Section~\ref{sec:projection} to the setting of populations. We start
with the definition of a lower diamond.

\begin{definition} \label{def:lower_dia_pop} For $\set{\Pop_1,
      \dots, \Pop_m} \subset \bbP$, we call
    $(\Pop^\wedge,\Pop_\vee;\Pop_1, \dots, \Pop_m)$ a {\em lower
      diamond} in $\bbP$ of size $m$, when
    $\Pop^\wedge=\bigwedge_{j=1}^m \Pop_j$, $\Pop_\vee=\bigvee_{j=1}^m
    \Pop_j$, and the sets $\Pop_\vee \setminus \Pop_j$ are mutually
    disjoint.
\end{definition}  

As we shall see in Corollary~\ref{cor:equiv-of-diamonds} below,
Definition~\ref{def:lower_dia_pop} in terms of populations is
consistent with 
Definition~\ref{def:lower_dia_tria} in terms of triangulations.  In particular, all results of
Section~\ref{sec:projection} dealing with lower diamonds transfer to populations.

Recall from Subsection~\ref{sec:populations} the definition
of $\Omega(P)$ for $P \in \Pop^\top \setminus \Pop_\bot$.
\begin{lemma}
  \label{lem:omega_disjoint}
  Let $P_1, P_2 \in \Pop^\top \setminus \Pop_\bot$ with $\Omega(P_1)
  \cap \Omega(P_2) \neq \emptyset$. Then either $P_1 = P_2$ or
  $P_1 \ancestorof P_2$ or $P_2 \ancestorof P_1$.
\end{lemma}
\begin{proof}
  If $\generation(P_1)=\generation(P_2)$, then the claim follows by
  Proposition~\ref{prop:areas}\,\ref{areas_item1}. W.l.o.g. assume now
  that $\generation(P_1) < \generation(P_2)$. By (a repeated)
  application of Proposition~\ref{prop:areas}\,\ref{areas_item3}, we
  have
  \begin{align*} 
   \overline{\Omega(P_2)} \subset \overline{\Omega( \ancestor(P_2) \cap
    \generation^{-1}(\generation(P_1)))}.
  \end{align*}
  Thus $\Omega(P_1) \cap \Omega( \ancestor(P_2) \cap
  \generation^{-1}(\generation(P_1)) \neq \emptyset$, and therefore there
  exists a $P_3 \ancestorof P_2$ with $\generation(P_3) =
  \generation(P_1) $ and $\Omega(P_1) \cap \Omega(P_3)\neq \emptyset$.
  Finally, Proposition~\ref{prop:areas}\,\ref{areas_item1} implies $P_1 =
  P_3$, i.e., $P_1 \ancestorof P_2$.
\end{proof}

\begin{lemma}
  \label{lem:pw_disjoint}
  Let $\Pop_1,\Pop_2,\Pop_* \in \bbP$ with $\Pop_1,\Pop_2 \leq
  \Pop_*$. Define $\mathcal{R}_j := \Pop_* \setminus \Pop_j$ for
  $j=1,2$. Then $\mathcal{R}_1 \cap \mathcal{R}_2 =\emptyset$ if and
  only if $\Omega(\mathcal{R}_1) \cap
  \Omega(\mathcal{R}_2)=\emptyset$.
\end{lemma}

\begin{proof}
  Assume that $\mathcal{R}_1\cap \mathcal{R}_2 \neq \emptyset$. Then
  for $P \in \mathcal{R}_1\cap \mathcal{R}_2$ we obviously have
  $\emptyset \neq \Omega(P) \subset \Omega(\mathcal{R}_1) \cap
  \Omega(\mathcal{R}_2)$.

  Assume now that $\Omega(\mathcal{R}_1) \cap \Omega(\mathcal{R}_2)
  \neq \emptyset$.  Therefore, there exists $P_1\in \mathcal{R}_1$ and $P_2
  \in \mathcal{R}_2$ with $\Omega(P_1) \cap \Omega(P_2) \neq
  \emptyset$.  It follows from Lemma~\ref{lem:omega_disjoint} that
  either $P_1= P_2$, and thus ${\mathcal R}_1\cap {\mathcal R}_2 \neq
  \emptyset$, or $P_1 \ancestorof P_2$ or $P_2 \ancestorof P_1$.  In the
  case $P_1\ancestorof P_2$ we obtain from  \new{$P_1\in \mathcal{R}_1$ that  $P_1\not\in \Pop_1$, and thus $P_2
  \not\in \Pop_1$ since $\Pop_1\in \bbP$}.  Consequently, we have $P_2 \in {\mathcal
    R}_1\cap {\mathcal R}_2$.  The same argument shows that ${\mathcal
    R}_1\cap {\mathcal R}_2\neq \emptyset$ when $P_2 \ancestorof P_1$.
\end{proof}

For $\Pop,\Pop_* \in \bbP$, we have $\Omega(\Pop_* \setminus
\Pop)=\Omega(\tria(\Pop_*) \setminus \tria(\Pop))$. Hence, as an immediate
consequence we obtain the following result. 
\begin{corollary} \label{cor:equiv-of-diamonds}
$(\Pop^\wedge,\Pop_\vee;\Pop_1,\ldots,\Pop_m)$ is a lower diamond in
$\bbP$ if and only if  
$(\tria(\Pop^\wedge),\tria(\Pop_\vee);\tria(\Pop_1),\ldots,\tria(\Pop_m))
$ is a lower diamond in $\bbT$. 
\end{corollary}

This allows us to reformulate the lower diamond estimate in terms of
populations. In particular,  Corollary~\ref{cor:G_LDE} reads as:

\begin{corollary}[Lower diamond estimate] 
 \label{cor:LDE-G}
Let $(\Pop^\wedge,\Pop_\vee;\Pop_1,\ldots,\Pop_m)$ be a lower diamond in $\bbP$, then
 \begin{align*} 
   \mathcal{G}(\Pop^\wedge) - \mathcal{G}(\Pop_\vee) \eqsim
   \sum_{j=1}^m \big( \mathcal{G}(\Pop_j) -
   \mathcal{G}(\Pop_\vee)\big).
 \end{align*}
\end{corollary}

\section{Energy optimality and instance optimality}
\label{sec:simpl-optim}

For each $m \in \setN_0$,  we define the minimal energy level of 
populations with not more than $\#\Pop_\bot+m$ persons by
\begin{align*}
  \Gopt[m] =\min\set{\mathcal{G}(\Pop)\,:\,\Pop \in \bbP,\, \#(\Pop
    \setminus \Pop_\bot) \leq m}.
\end{align*}
Since the set on the right-hand side is finite, the minimum is attained
and there exists a population $\Popt[m] \in \bbP$ such that $\Gopt[m]
= \mathcal{G}(\Popt[m])$.  Our analysis does not rely on the
particular choice of $\Popt[m]$ and therefore we may ignore the fact
that the choice of $\Popt[m]$ may not be unique.

The AFEM, Algorithm~\ref{algo:AFEM}, produces a monotone increasing
sequence $\Pop_k$
of populations with $\Pop_0 = \Pop_\bot$.  We say that the AFEM
is {\em energy  optimal}, when there exists a constant~$C>0$
such that 
$\mathcal{G}(\Pop_k) \leq \Gopt[m]$, whenever $\#(P_k \setminus
\Pop_\bot) \geq C\,m$.

\begin{lemma} 
  \label{lem:optimal_single_level}
   Consider the sequences $(\Pop_k)_{k \in \mathbb{N}_0}$ and
  $(\mathcal{M}_k)_{k \in \mathbb{N}_0}$ produced by
  Algorithm~\ref{algo:AFEM} (second formulation). Then there exists
  \new{a constant} \newer{$\gamma\Cgeq \mu$}, 
  such that: 
  If, for some $k, m \in \mathbb{N}_0$, we have $\Gopt[m] 
  \geq \mathcal{G}(\Pop_k) > \Gopt[m+1]$,
  then
  \begin{align*}
    \mathcal{G}(\Pop_k) - \mathcal{G}(\Pop_{k+1}) &\geq \gamma \,\#\Mk
    \, \big(\Gopt[m] - \Gopt[m+1] \big).
      \end{align*}
\end{lemma}

\begin{proof}
Let $k, m \in \mathbb{N}_0$ be such that $\Gopt[m] 
  \geq \mathcal{G}(\Pop_k) > \Gopt[m+1]$.
  \begin{figure}[h]
    \centering
    \input{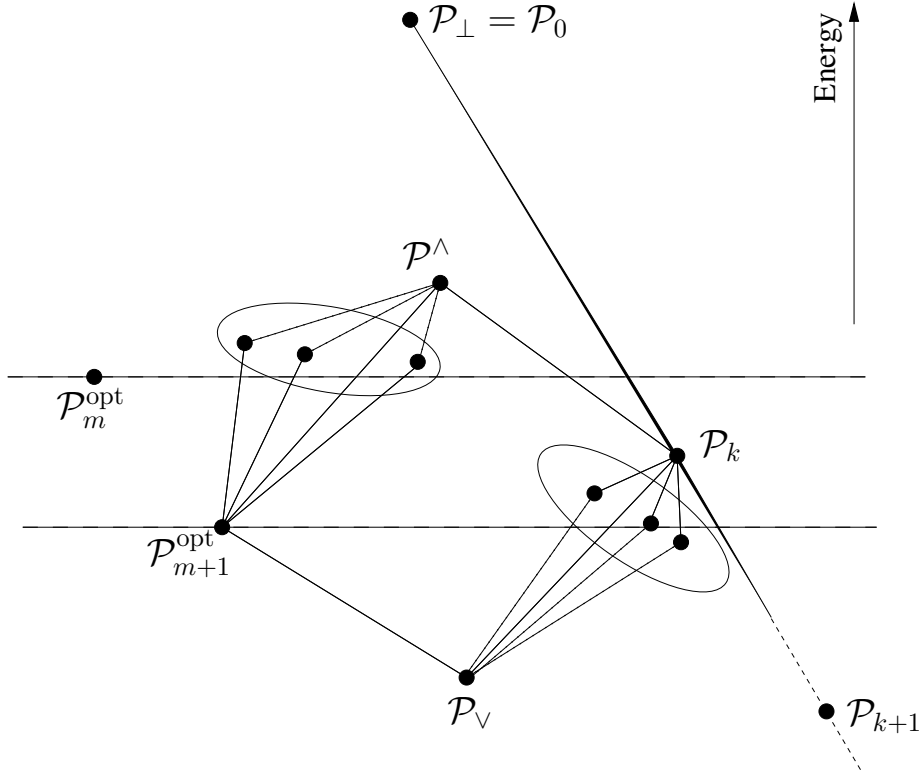}
    \caption{Illustration with the proof of
      Lemma~\ref{lem:optimal_single_level}. The populations
      represented by the dots in the left and right ellipses are 
      $\set{\Popt[m+1] \ominus P: P \in \mathcal{C}}$ and $\set{\Pop_k \oplus P: P \in
        \mathcal{U}}$,
      respectively.}
    \label{fig:opt}
  \end{figure}
  From Corollary~\ref{C:energy_gain_P} and
  Proposition~\ref{prop:lwb}, we obtain
  \begin{equation} \label{eq:opt0} \mathcal{G}(\Pop_k) -
    \mathcal{G}(\Pop_{k+1}) \Cgeq \mathcal{E}^2_{\Pop_k}(\Pop_{k+1}
    \cap (\Pop_k\plus\setminus \Pop_k)) \geq \mu\,\# \Mk \,\estPbar[\Pop_k]^2.
  \end{equation}  
  
  Setting $\Pop_\vee :=
  \Pop_k\vee \Popt[m+1]$ and $\Pop^\wedge := \Pop_k \wedge
  \Popt[m+1]$,  we choose
  $$
  \mathcal{U} := \free(\Pop_\vee \cap (\Pop_k\plus\setminus \Pop_k));
  $$ 
 see Figure~\ref{fig:opt}.
  Since
  $\mathcal{G}(\Popt[m+1])<\mathcal{G}(\Pop_k)$, we have
  $\mathcal{U}\neq \emptyset$; compare with Lemma~\ref{lem:free}\ref{itm:free3}.
  Thanks to the definition of $\estPbar[\Pop_k]^2$ \new{(see Algorithm~\ref{algo:AFEM})} and $\mathcal{U}
  \subset \Pop_k\plus \setminus \Pop_k$ we obtain that
  \begin{align*}
    \estPbar[\Pop_k]^2 &\geq \frac{1}{\#\mathcal{U}} \sum_{P \in
      \mathcal{U}} \mathcal{E}_{\Pop_k}^2((\Pop_k \oplus P)\setminus
    \Pop_k) \geq \frac{1}{\#\mathcal{U}} \,
    \mathcal{E}_{\Pop_k}^2\Big(\bigcup_{P \in \mathcal{U}} (\Pop_k
    \oplus P)\setminus \Pop_k\Big).
  \end{align*}
  By Lemma~\ref{lem:free}\,\ref{itm:free2} we have $\mathcal{U} \cup
  \ancestor(\mathcal{U}) \supset \Pop_\vee \cap (\Pop_k\plus\setminus
  \Pop_k)$, and thus
  \begin{align*}
    \bigcup_{P \in \mathcal{U}} (\Pop_k \oplus P) \setminus
    \Pop_k &= \bigcup_{P \in \mathcal{U}} \big(\set{P} \cup
    \ancestor(P)\big) \setminus \Pop_k
    \\
    &= \big(\mathcal{U} \cup \ancestor(\mathcal{U})\big) \setminus
    \Pop_k
    \\
    &\supset \big(\Pop_\vee \cap (\Pop_k\plus \setminus \Pop_k)) \setminus
    \Pop_k
    \\
    &= \Pop_\vee \cap (\Pop_k\plus \setminus \Pop_k).
  \end{align*}
  This and the previous estimate prove
  \begin{equation*} 
    \estPbar[\Pop_k]^2 \geq
    \frac{1}{\#\mathcal{U}} \,\mathcal{E}_{\Pop_k}^2\big(\Pop_\vee \cap
    (\Pop_k\plus \setminus \Pop_k)\big).
  \end{equation*}
  An application of Corollary~\ref{C:energy_gain_P} then shows that
  \begin{equation*} 
    \mathcal{E}_{\Pop_k}^2(\Pop_\vee
    \cap (\Pop_k\plus\setminus \Pop_k)) \Cgeq \mathcal{G}(\Pop_k) -
    \mathcal{G}(\Pop_\vee).
  \end{equation*}
  
  Since $\Pop_k \neq \Popt[m+1]$, by Lemma~\ref{lem:lowersize2} and Corollary~\ref{cor:equiv-of-diamonds} we have that
    $(\Pop^\wedge,\Pop_\vee;\Pop_k,\Popt[m+1])$ is a lower diamond in $\bbP$.
  By the lower diamond estimate Corollary~\ref{cor:LDE-G} together with
  ${\mathcal G}(\Pop_k) \geq
      {\mathcal G}(\Popt[m+1]) \geq {\mathcal G}(\Pop_\vee)$, this implies
  \begin{align*} 
   \begin{aligned}
      \mathcal{G}(\Pop_k) - \mathcal{G}(\Pop_\vee)& \geq \frac 12 \Big(
      \big( \mathcal{G}(\Pop_k) - \mathcal{G}(\Pop_\vee)\big) + \big(
      \mathcal{G}(\Popt[m+1]) - \mathcal{G}(\Pop_\vee)\big) \Big)
      \\
      & \eqsim \mathcal{G}(\Pop^\wedge) - \mathcal{G}(\Pop_\vee) \geq  \mathcal{G}(\Pop^\wedge) - \mathcal{G}(\Popt[m+1]).
    \end{aligned}
  \end{align*}
  Combining the above observations with \eqref{eq:opt0}, yields
  \begin{equation} \label{eq:opt5}
    \begin{split}
      \mathcal{G}(\Pop_k) - \mathcal{G}(\Pop_{k+1}) &\Cgeq \mu\,\frac{ \#
        \Mk}{\#\mathcal{U}} \, \mathcal{E}_{\Pop_k}^2(\Pop_\vee \cap
      (\Pop_k\plus\setminus \Pop_k))
      \\
      & \Cgeq \mu\,\frac{ \# \Mk}{\#\mathcal{U}} \,
      \big(\mathcal{G}(\Pop^\wedge) - \mathcal{G}(\Popt[m+1])\big).
    \end{split}
  \end{equation}

  For every $P \in \mathcal{C} := \free( \Popt[m+1] \setminus
  \Pop^\wedge) \neq \emptyset$, we set $\Pop'_P := \Popt[m+1] \ominus P$; see Figure~\ref{fig:opt}. 
  Thanks to Lemma~\ref{lem:free_removed}, we know $\Pop'_P= \Popt[m+1]
  \setminus \set{P}$,
  and thus $\Pop^\wedge \leq \new{\bigwedge}_{P \in \mathcal{C}} \Pop'_{P}$.
  If $\#\mathcal{C}>1$, then the sets 
  $\Popt[m+1] \setminus \Pop'_P$ for $P \in \mathcal{C}$ are mutually disjoint, and $\bigvee_{P \in \mathcal{C}} \Pop'_{P}=\Popt[m+1]$.
  Applying the lower diamond estimate to the lower diamond
    $(\bigwedge_{P \in \mathcal{C}} \Pop'_{P},\Popt[m+1];(\Pop'_P)_{P
  \in \mathcal{C}})$ in $\bbP$, we obtain that
\begin{align*}
  \mathcal{G}(\Pop^\wedge) - \mathcal{G}(\Popt[m+1]) 
  & \geq 
  \mathcal{G}(\bigwedge_{P' \in \mathcal{C}} \Pop'_{P'}) -
  \mathcal{G}(\Popt[m+1] ) 
  \eqsim
  \sum_{P \in \mathcal{C}} \big(\mathcal{G}(\Pop'_P)-\mathcal{G}(\Popt[m+1])\big).
\end{align*}
If $\#\mathcal{C}=1$, then the last step is obvious, whence the estimate is true in
any case. Since $\#(\Pop'_P \setminus \Pop_\bot) = \#(\Popt[m+1]
  \setminus\Pop_\bot)-1 \leq m$, we 
  have
  \begin{align*}
    \mathcal{G}(\Pop'_P) &\geq 
    \mathcal{G}(\Popt[m]) .
  \end{align*}
  Therefore, we conclude from \eqref{eq:opt5}, that 
  \begin{align}
    \label{eq:opt8}
    \mathcal{G}(\Pop_k) - \mathcal{G}(\Pop_{k+1}) &\Cgeq \mu\,\frac{\#
      \Mk \#\mathcal{C}}{\#\mathcal{U}} \,\big(
    \mathcal{G}(\Popt[m]) - \mathcal{G}(\Popt[m+1]) \big).
  \end{align}

  It remains to prove that $\#\mathcal{C} \gtrsim \#\mathcal{U}$.
  Define $\mathcal{V} := \Pop_\vee \cap (\Pop_k\plus \setminus
  \Pop_k)$, then $\mathcal{U} = \free(\mathcal{V})$ and $\mathcal{V}
  \subset \Pop_\vee \setminus \Pop_k$.  Since $\Pop_\vee \setminus
  \Pop_k = (\Pop_k \cup \Popt[m+1])\setminus \Pop_k= \Popt[m+1]
  \setminus (\Pop_k \cap \Popt[m+1]) = \Popt[m+1] \setminus
  \Pop^\wedge$, we have~$\mathcal{V} \subset \Popt[m+1] \setminus
  \Pop^\wedge$. Thus Theorem~\ref{thm:free_bound} and $\mathcal{C}
  = \free( \Popt[m+1]\setminus \Pop^\wedge)$ yield
  \begin{align}
    \label{eq:opt9}
    \#\mathcal{U} = \# \free(\mathcal{V}) \leq \cGD \,\#
    \free(\Popt[m+1] \setminus \Pop^\wedge ) = \cGD \,\# \mathcal{C}.
  \end{align}
  Therefore,~\eqref{eq:opt8} and \eqref{eq:opt9} imply the desired estimate
  \begin{align*}
    \mathcal{G}(\Pop_k) - \mathcal{G}(\Pop_{k+1}) &\gtrsim \mu\,\#
    \Mk \,\big( \mathcal{G}(\Popt[m]) -
    \mathcal{G}(\Popt[m+1]) \big). \qedhere
  \end{align*}
\end{proof}

\new{Instance optimality of the AFEM would follow from
  Lemma~\ref{lem:optimal_single_level} when there would be a uniform
  bound on the cardinalities of the sets $\mathcal{M}_k$ of marked
  edges. Such a bound, however, does not exist, and the case of having
  `many' marked edges will be covered by the following lemma.}

\begin{lemma}
  \label{lem:optimal_multi_level}
  Consider the sequences $(\Pop_k)_{k \in \mathbb{N}_0}$ and
  $(\mathcal{M}_k)_{k \in \mathbb{N}_0}$ produced by
  Algorithm~\ref{algo:AFEM} (second formulation). Then
  there exists a constant \newer{$1\le K \Cleq 1/\mu$} 
  such that: 

  If, for some $k,m \in \mathbb{N}_0$, $\Gopt[m] 
  \geq \mathcal{G}(\Pop_k) > \Gopt[m+1]$,
  then
  \begin{align*}
    \mathcal{G}(\Pop_{k+1}) & \leq
    \Gopt[m+\lfloor \frac{\# \Mk}{K}
    \rfloor].
  \end{align*}
\end{lemma}

\begin{proof} 
  For some $k, m \in \mathbb{N}_0$, let  $\Gopt[m] 
  \geq \mathcal{G}(\Pop_k) > \Gopt[m+1]$.
  For $\# \Mk <K$, we have that $\lfloor \frac{\# \Mk}{K} \rfloor=0$, and the
  claim is a direct consequence of the monotonicity of the total energy.
  Therefore, assume that $\# \Mk \geq K$.  We set $\alpha:=\lfloor
  \frac{\# \Mk}{K} \rfloor$.   
  Setting
  $\Pop_\vee := \Pop_k
  \vee \Popt[m+\alpha]$ and $\Pop^\wedge := \Pop_k \wedge
  \Popt[m+\alpha]$.
  We repeat the steps in the proof of
  Lemma~\ref{lem:optimal_single_level} up to~\eqref{eq:opt5}, 
  now using the diamond
  $(\Pop^\wedge,\Pop_\vee;\Pop_k,\Popt[m+\alpha])$ in $\bbP$.
  Then, for
  $\mathcal{U}:=\free(\Pop_\vee \cap (\Pop_k\plus \setminus \Pop_k))$
  we get
  \begin{align}
    \label{eq:opt10}
    \mathcal{G}(\Pop_k) - \mathcal{G}(\Pop_{k+1}) &\Cgeq \frac{\mu\,\#
      \Mk}{\# \mathcal{U}} \big(\mathcal{G}(\Pop^\wedge) -
    \mathcal{G}(\Popt[m+\alpha]) \big).
  \end{align}
  
  We define $\mathcal{C} := \free( \Popt[m+\alpha] \setminus
  \Pop^\wedge)$, and set $N:=\lfloor \frac{\#
    \mathcal{C}}{\alpha}\rfloor$.  Exactly as in
  Theorem~\ref{thm:free_bound} equation \eqref{eq:opt9}, we have
  $\# \mathcal{U} \leq \cGD \# \mathcal{C}$.

  If $N=0$, then $\# \mathcal{C} < \frac{\# \Mk}{K}$, and hence
  $\frac{\# \Mk}{\# \mathcal{U}} > \frac{K}{\cGD}$.  \newer{Thus, for every 
  $K\ge \cGD/\mu$, we have} that 
  $\mathcal{G}(\Pop_k) - \mathcal{G}(\Pop_{k+1}) \geq
  \mathcal{G}(\Pop^\wedge) - \mathcal{G}(\Popt[m+\alpha])$, and thus
  by $\mathcal{G}(\Pop^\wedge) \geq \mathcal{G}(\Pop_k)$ we arrive at
  $\mathcal{G}(\Pop_{k+1}) \leq \mathcal{G}(\Popt[m+\alpha])$.

  Else, if $N \geq 1$, then $N \geq \frac{\# \mathcal{C}}{2\alpha}$.
  For $\mathcal{C}_1, \dots, \mathcal{C}_N$ being mutually disjoint
  subsets of~$\mathcal{C}$, each having~$\alpha$ elements, we set $\Pop'_j :=
  \Popt[m+\alpha] \ominus \mathcal{C}_j$, $j=1,\dots,N$. It follows
  from Lemma~\ref{lem:free_removed}, that $\Pop'_j = \Popt[m+\alpha]
  \setminus \mathcal{C}_j$, and hence we have $\Pop^\wedge \leq
  \bigwedge_{j=1}^N \Pop'_j $ and $\# (\Pop'_j \setminus \Pop_\bot)
  \leq 
  m$. The last inequality implies that
  \begin{align*}
        \mathcal{G}(\Pop'_j) &\geq
    \mathcal{G}(\Popt[m]) , \qquad
    j=1,\dots,N.
  \end{align*}  
  If $N>1$, then the sets 
  $\Popt[m+\alpha] \setminus \Pop'_j\new{=\mathcal{C}_j}$ for $1 \leq j \leq N$ are mutually disjoint, and $\bigvee_{j=1}^N \Pop'_{j}=\Popt[m+\alpha]$.
  By
    applying the lower diamond estimate to the lower diamond
    $(\bigwedge_{j=1}^N \Pop'_{j},\Popt[m+\alpha];(\Pop'_j)_{1 \leq j \leq N})$ in $\bbP$, we obtain that
\begin{align*}
  \mathcal{G}(\Pop^\wedge) - \mathcal{G}(\Popt[m+\alpha]) 
  & \geq 
  \mathcal{G}(\bigwedge_{j=1}^N \Pop'_j) -
  \mathcal{G}(\Popt[m+\alpha] ) 
  \eqsim
  \sum_{j=1}^N \big(\mathcal{G}(\Pop'_j)-\mathcal{G}(\Popt[m+\alpha])\big).
\end{align*}
If $N=1$, then this estimate is obvious, whence it is true for $N \geq 1$.

   Therefore, we can further estimate  \eqref{eq:opt10} by
  \begin{align*}
    \mathcal{G}(\Pop_k) - \mathcal{G}(\Pop_{k+1}) &\Cgeq \frac{\mu N \#
      \Mk}{\# \mathcal{U}} \big(\mathcal{G}(\Popt[m]) -
    \mathcal{G}(\Popt[m+\alpha])\big)\\
    & \geq \frac{\mu K}{2\cGD}\big(\mathcal{G}(\Popt[m]) -
    \mathcal{G}(\Popt[m+\alpha])\big).
  \end{align*}  
  In the last estimate, we have used that $\frac{ N \# \Mk}{\# \mathcal{U}}
  \geq \frac{K}{2 \cGD}$, since $N \geq \frac{\# \mathcal{C}}{2
    \alpha}$, $\frac{1}{\alpha} \geq \frac{K}{ \# \Mk}$, and
  $\frac{1}{\# \mathcal{U}} \geq \frac{1}{\cGD \#\mathcal{C}}$.  By
  taking \newer{$K= 2\cGD/\mu$}, we conclude that 
  $\mathcal{G}(\Pop_k) - \mathcal{G}(\Pop_{k+1}) \geq
  \mathcal{G}(\Popt[m]) - \mathcal{G}(\Popt[m+\alpha])$. Thus
  from $\mathcal{G}(\Popt[m]) \geq \mathcal{G}(\Pop_k)$, we 
  obtain that
  $\mathcal{G}(\Pop_{k+1}) \leq \mathcal{G}(\Popt[m+\alpha])$.  This
  proves the claim in the case $N \geq 1$.
\end{proof}

\begin{theorem}
  \label{thm:energy_opt}
  Consider the sequences $(\Pop_k)_{k \in \mathbb{N}_0}$ and
  $(\mathcal{M}_k)_{k \in \mathbb{N}_0}$ produced by
  Algorithm~\ref{algo:AFEM} (second formulation).
  Then, there exists a constant $1\le C \newer{\Cleq 1/\mu^2}$, 
  such 
  that $\#(\Pop_k \setminus \Pop_\bot) \geq C m$ implies
  $\mathcal{G}(\Pop_k) \leq \Gopt[m]$, i.e., 
    the algorithm is {\em energy optimal} with respect to the total 
  energy~$\mathcal{G}$. 
\end{theorem}

\begin{proof}
  Let $\gamma$ and $K$ be the constants
  from Lemmas~\ref{lem:optimal_single_level} and
  \ref{lem:optimal_multi_level}. Setting
  $$
  C_\ell:=\sum_{j=0}^{\ell-1} \#\Mk[j], \quad R:= \Big\lceil
  \frac{1}{\gamma}\Big\rceil, \quad L:=2(R-1)(K-1)+2K,
  $$
  the claim follows from
  \begin{equation} \label{eq:opt17}
    \mathcal{G}(\Pop_k) \leq \mathcal{G}(\Popt[\lfloor C_k/L\rfloor]).
  \end{equation}
  Indeed, Corollary~\ref{cor:BDD} and \eqref{eq:no_P_vs_T} imply
  that there exists some constant $D>0$, depending solely on
  $\tria_\bot$, such that
  $\#(\Pop_k\setminus \Pop_\bot) \leq D\, C_k $.  Taking $C:=2DL$, we
  conclude that if $\#(\Pop_k\setminus \Pop_\bot) \geq C m$, then
  $\frac{C_k}{2L} \geq m$, and thus $\lfloor C_k / L \rfloor \geq m$,
  which shows $\mathcal{G}(\Pop_k) \leq \mathcal{G}(\Popt[m])$ by
  \eqref{eq:opt17}.  \newer{Note that we have $C\Cleq 1/\mu^2$, thanks
    to the dependence of the constants $\gamma$
    and $K$ on the marking parameter $\mu$; compare with  Lemmas~\ref{lem:optimal_single_level} and~\ref{lem:optimal_multi_level}.} 
  
  We shall prove \eqref{eq:opt17} by
  induction. Obviously, \eqref{eq:opt17}  is valid for $k=0$. Fixing an arbitrary $k \in
  \mathbb{N}$, assume that 
  \eqref{eq:opt17} is valid for $\mathbb{N}_0 \ni
  k' <k$.  Since
  \eqref{eq:opt17} is obviously true when
  $\mathcal{G}(\Pop_k)=\mathcal{G}(\Pop^\top)$, we assume that
  $\mathcal{G}(\Pop_k)>\mathcal{G}(\Pop^\top)$.
  
  First, assume that the set 
  \begin{align}
    \label{eq:opt19}\big\{\ell' \in \set{\max(k-R,0)\new{,\max(k-R+1,0)},\ldots,k-1}: \#\Mk[\ell'] \geq K\big\} 
  \end{align}
  is non-empty and set $\ell$ to be its maximal element.
  \new{By the induction hypothesis, there exists an $m \geq \lfloor C_\ell / L \rfloor$ such that
  $\mathcal{G}(\Popt[m+1]) <\mathcal{G}(\Pop_\ell) \leq
  \mathcal{G}(\Popt[m] )$.  By
  Lemma~\ref{lem:optimal_multi_level} we obtain that} $\mathcal{G}(\Pop_{k})
  \leq \mathcal{G}(\Pop_{\ell+1}) \leq \mathcal{G}(\Popt[m+\lfloor  \#
  {\Mk[\ell]}/K\rfloor])$.  Using that $\lfloor a \rfloor + \lfloor b
  \rfloor \geq \lfloor a +b/2 \rfloor$ for $b \geq 1$, $\#
  \Mk[\ell]\geq K$, the definition of $L$, and $\# \Mk[\ell'] \leq K-1$
  for the at most $R-1$ integers $\ell < \ell' \leq k-1$, we arrive at
  \begin{align*}
    m+\lfloor \# \Mk[\ell]/K\rfloor  &\geq \lfloor C_\ell / L \rfloor+
    \lfloor  \# {\Mk[\ell]}/K\rfloor \geq 
    \lfloor C_\ell / L + \# {\Mk[\ell]}/(2K)\rfloor
    \\
    &\geq \lfloor (C_\ell+ \#{\Mk[\ell]} +(R-1)(K-1)) / L \rfloor
    \geq \lfloor C_{k} / L \rfloor.
  \end{align*}
  This completes the proof of \eqref{eq:opt17} in the case that the
  set in \eqref{eq:opt19} is non-empty.
  
  Suppose now, that the set in \eqref{eq:opt19} is empty.  If
  $k<R$, then $C_{k} \leq (R-1)(K-1)<L$, and we have
  \eqref{eq:opt17}.  Now let $k \geq R$.
  \new{By the induction hypothesis, there exists an $m \geq \lfloor
  C_{k-R} / L \rfloor$} such that $\mathcal{G}(\Popt[m+1])
  <\mathcal{G}(\Pop_{k-R}) \leq \mathcal{G}(\Popt[m] )$.  By a
  repeated application of Lemma~\ref{lem:optimal_single_level} with
  $k$ reading as $k-R, k-R+1,\ldots,\ell$, as long as
  $\mathcal{G}(\Pop_\ell) >\mathcal{G}(\Popt[m+1])$, we find that
  \begin{align*}
    \mathcal{G}(\Pop_{k-R})-
    \mathcal{G}(\Pop_{\ell+1})&=\mathcal{G}(\Pop_{k-R})-
    \mathcal{G}(\Pop_{k-R+1})+ \cdots +\mathcal{G}(\Pop_{\ell})-
    \mathcal{G}(\Pop_{\ell+1})
    \\
    &\geq \gamma (\ell-k+R+1)
    (\mathcal{G}(\Popt[m])-\mathcal{G}(\Popt[m+1]))
    \\
    & \geq \gamma (\ell-k+R+1)
    (\mathcal{G}(\Pop_{k-R})-\mathcal{G}(\Popt[m+1]).
  \end{align*}
  Therefore, by definition of $R$, for $\ell=k-1$ at the latest it
  holds that $\mathcal{G}(\Pop_{\ell+1}) \leq
  \mathcal{G}(\Popt[m+1])$, and thus $\mathcal{G}(\Pop_{k}) \leq
  \mathcal{G}(\Popt[m+1])$.  Since $L \geq R(K-1)$ and
  $\#\Mk[\ell'] \leq K-1$ for $\ell' \in \set{k-R,\ldots,k-1}$, we
  have
  \begin{align*}
    m+1 \geq \lfloor 1+ C_{k-R} / L \rfloor \geq \lfloor (R(K-1)+
    C_{k-R}) / L \rfloor \geq \lfloor C_{k} / L \rfloor.
  \end{align*}
  This proves \eqref{eq:opt17}.
\end{proof}

We are now ready to prove {\em instance optimality} of our AFEM as
was announced in the introduction: 

\begin{theorem}
  \label{thm:instance}
  There  exist constants
  \newer{$1\le C\Cleq 1/\mu^2$} and $\tilde C \new{\geq 1}$, such that for $(\tria_k)_{k \in 
    \mathbb{N}_0} \subset \bbT$ being the sequence of triangulations
  produced by Algorithm~\ref{algo:AFEM} (first formulation), it holds that 
  \begin{align*}
    \abs{u-\uT[\tria_k]}_{H^1(\Omega)}^2+
    \osc_{\tria_k}^2(\new{\tria_k}) \le \tilde C\, \big(
    \abs{u-\uT[\tria]}_{H^1(\Omega)}^2+\osc_\tria^2(\new{\tria})
    \big)
  \end{align*}
  for all $\tria\in\bbT$ with \new{$\#(\tria \setminus \tria_\bot)\le \frac{\#(\tria_k \setminus \tria_\bot)}{C}$}. 
  The constant $\tilde{C}$ \new{depends only possibly} on $\tria_\bot$. 
\end{theorem}

\begin{proof}
  We know from~\eqref{eq:energy_total_error} that the total
  energy~$\mathcal{G}$ satisfies
  \begin{align*}
    \mathcal{G}(\new{\tria})-\mathcal{G}(\new{\tria^\top})\eqsim
    \abs{u-\uT[\tria]}^2_{H^1(\Omega)}+\osc_\tria^2(\new{\tria}) 
  \end{align*}
  for all $\tria\in\bbT$. Hence the assertion follows from 
  Theorem~\ref{thm:energy_opt}.
\end{proof}

\appendix
\new{
\section{A slightly modified marking} \label{appendix}

In this section we propose a routine \textsf{MARK}, resorting to slightly modified
accumulated indicators, that can be implemented in  ${\mathcal O}(\#\Pop_k)$ operations.  
The important fact is that Proposition~\ref{prop:lwb} remains valid, which ensures the
instance optimality in Section~\ref{sec:simpl-optim} also for this modified marking step.

To this end, we first compute a modified maximal accumulated
indicator $\EstPbar[\Pop_k]^2$. 
This value can be determined 
with the help of the following recursive routine. 
\begin{algorithm}[Maximal Indicator]
  Set $\EstPbar[\Pop_k]^2:=0$ and call $\textsf{max-ind}(P,0)$ for all
$P\in(\Pop_k\plus\setminus\Pop_k)$ with no parents in $\Pop_k\plus\setminus
  \Pop_k$.
  {\em \par{ \linespread{1.5} \sffamily \begin{tabbing} \mbox{}\quad\=
      \quad \= \quad \=\kill 
       \>$\text{procedure \sf max-ind}(P,\text{value-parent})$
      \\
      \>\>$\EstPbar[\Pop_k]^2 := \max\{\EstPbar[\Pop_k]^2,\estP[\Pop_k]^2(P) + \text{value-parent}\}$;
      \\
      \>\>for each child $C \in \children(P) \cap(\Pop\plus_k \setminus
      \Pop_k)$ do
      \\
      \>\>\> $\textsf{max-ind}(C,\estP[\Pop_k]^2(P) +
      \text{value-parent})$;
      \\
      \>\>end for
      \\
      \>end procedure \textsf{max-ind}
\end{tabbing}}}
\end{algorithm}
In general we have $\EstPbar[\Pop_k]^2\neq
\estPbar[\Pop_k]^2$, since
$C\in \Pop_k\plus \setminus
\Pop_k$ may have {\em two} parents $P,P' \in \Pop_k\plus\setminus
\Pop_k$.  However, such a $C$ cannot have children in $\Pop\plus
\setminus \Pop$ as is illustrated in Figure~\ref{markfig},
and so we conclude that 
\begin{align*}
  \EstPbar[\Pop_k]^2\ge \frac12\estPbar[\Pop_k]^2.
\end{align*}

\begin{figure}
\input{mark.pspdftex}
\caption{$\tria(\Pop)$ (solid lines), $P,P',Q \in \Pop\plus\setminus
  \Pop=\midpoint{\sides(\tria(\Pop))}$ with $\{P,P'\} = \parents(Q)$,
  and $\{\Box\}=\children(Q).$} 
\label{markfig}
\end{figure}

Next, the sets $\Mk$ and $\tildeMk$ are determined
by running the following routine. 

 \begin{algorithm}[Marking]\label{algo:mark}
   Set $\Mk:=\tildeMk:=\emptyset$ and call $\textsf{accum-est}(P,0)$ for
   all $P \in \Pop_k\plus \setminus\Pop_k$ with no parents in $\Pop_k\plus\setminus\Pop_k$.

{\em \par{ \linespread{1.5}
       \sffamily \begin{tabbing} \mbox{}\quad\= \quad \= \quad \=
         \quad \= \kill%
         \>$\text{boolean function \sf accum-est}(P,\text{value-parent})$
         \\
         \>\>$E_P^2:= \estP[\Pop_k]^2(P) +
         \text{value-parent}$;
         \\
         \>\>
         \\
         \>\> {\tt is\_marked} := false;
         \\
         \>\>if $E_P^2\geq \mu \estPbar[\Pop_k]^2$ then\\
         \>\>\> ${\mathcal M}_k:={\mathcal M}_k
         \cup \{P\}$; $\widetilde{\mathcal M}_k:=\widetilde{\mathcal M}_k
         \cup \{P\}$; \quad
         \\
         \>\>\> $E_P^2:= 0$;
         \\
         \>\>\> {\tt is\_marked} := true;
         \\
         \>\>end if;\\
         \>\>for each child $C \in \children(P) \cap(\Pop\plus_k
         \setminus \Pop_k)$ do
         \\
         \>\>\> if $\text{accum-est}(C, E_P^2)$ then \quad \%
         child is
         marked, so mark the parent\\
         \>\>\> \> $\widetilde{\mathcal M}_k:=\widetilde{\mathcal M}_k \cup \{P\}$;\\
         \>\>\>\> $E_P^2:= 0$;
         \\
         \>\>\>\> {\tt is\_marked} := true;
         \\
         \>\>\> end if\\
         \>\>end for
         \\
         \>\>return {\tt is\_marked};
	 \\
         \>end function accum-est
       \end{tabbing}}}
 \end{algorithm}

 One verifies that $\Mk, \tildeMk \subset \Pop_k\plus\setminus\Pop_k$, $\Pop_k \oplus \Mk = \Pop_k \cup \tildeMk$, and 
 \begin{align*}
   \estP[\Pop_k]^2 \big( (\Pop_k \oplus \Mk)\setminus\Pop_k\big) 
\geq \mu \,\# \Mk\, \EstPbar[\Pop_k]^2\ge \frac12\mu \,\# \Mk\, \estPbar[\Pop_k]^2;
 \end{align*}
i.e., Proposition~\ref{prop:lwb} is still valid.

Finally, the work needed for this evaluation of MARK scales linearly with $\# \Pop_k$.
Indeed, the number of times that a $P \in \Pop_k\plus \setminus
\Pop_k$ is accessed by the flow of computation is proportional to the
number of calls $\text{accum-est}(P,\cdot)$ 
(being one), plus the number of its children in $\Pop_k\plus \setminus
\Pop_k$ (being uniformly bounded). 
}

\providecommand{\href}[2]{#2}



\end{document}